\documentclass[11pts]{article}
\usepackage{fullpage,url,amsthm,amsmath}
\usepackage{float}
\usepackage{calc}
\usepackage{amsmath}
\usepackage{amssymb}
\usepackage{graphicx}
\usepackage{subfigure}
\usepackage{multirow}

\usepackage{amsfonts}
\usepackage{epstopdf}
\ifpdf
  \DeclareGraphicsExtensions{.eps,.pdf,.png,.jpg}
\else
  \DeclareGraphicsExtensions{.eps}
\fi

\usepackage{enumitem}
\setlist[enumerate]{leftmargin=.5in}
\setlist[itemize]{leftmargin=.5in}

\usepackage{amsopn}

\numberwithin{equation}{section}
\theoremstyle{plain}
\newtheorem{thm}{\protect\theoremname}[section]
  \theoremstyle{definition}
  \newtheorem{defn}[thm]{\protect\definitionname}
  \theoremstyle{definition}
  
  \theoremstyle{plain}
 \newtheorem{lem}[thm]{\protect\lemmaname}
  \theoremstyle{plain}
  \newtheorem{assumption}[thm]{\protect\assumptionname}
  \theoremstyle{remark}
  \newtheorem{rem}[thm]{\protect\remarkname}
  \theoremstyle{plain}
  \newtheorem{cor}[thm]{\protect\corollaryname}
  \theoremstyle{plain}
  \newtheorem{prop}[thm]{\protect\propositionname}

\usepackage[colorlinks,linkcolor=blue,citecolor=blue,urlcolor=blue]{hyperref}

  \providecommand{\assumptionname}{Assumption}
  \providecommand{\conditionname}{Condition}
  \providecommand{\corollaryname}{Corollary}
  \providecommand{\definitionname}{Definition}
  \providecommand{\lemmaname}{Lemma}
  \providecommand{\propositionname}{Proposition}
  \providecommand{\remarkname}{Remark}
 \providecommand{\theoremname}{Theorem}

\def\P{\mathbb P}
\def\R{\mathbb R}
\def\bfX{\mathbf X}
\def\bfZ{\mathbf Z}
\def\bfw{\mathbf w}
\def\bfx{\mathbf x}
\def\bfz{\mathbf z}
\def\bfu{\mathbf u}

\def\L{\mathcal L_\rho}
\def\U{\mathcal U_\rho}

\usepackage[normalem]{ulem}

\newcommand\comment[1]{}

\begin{document}


%
%

\title{Distributionally Robust Inference for Extreme Value-at-Risk\thanks{SS and DC were partially funded by the NSF grants DMS-1830293 and DMS-1243102, respectively.}}
\date{January 20, 2020}
\author{Robert Yuen\thanks{Seattle, WA 
  ({\tt bobbyyuen@gmail.com}).}
\and Stilian Stoev\thanks{{\bf Corresponding author.} Department of Statistics, The University of Michigan, Ann Arbor, MI 
  ({\tt sstoev@umich.edu}).}
\and Dan Cooley\thanks{Colorado State University, Fort Collins, CO ({\tt cooleyd@stat.colostate.edu}).}}

%
\maketitle

\begin{abstract}
Under general multivariate regular variation conditions, the extreme Value-at-Risk of a portfolio can be expressed as an integral of a known kernel with respect to a generally 
unknown spectral measure supported on the unit simplex. The estimation of the spectral measure is challenging in practice and  virtually impossible in high dimensions.  
This motivates the problem studied in this work, which is to find universal lower and upper bounds of the extreme Value-at-Risk under practically estimable constraints.  That 
is, we study the infimum and supremum of the extreme Value-at-Risk functional, over the infinite dimensional space of all possible spectral measures that meet a finite set of 
constraints. We focus on extremal coefficient constraints, which  are popular and easy to interpret in practice.  Our contributions are twofold.  First, we show that 
optimization problems over an infinite dimensional space of spectral measures are in fact dual problems to linear semi-infinite programs (LSIPs)
-- linear optimization problems in Euclidean space with an uncountable set of linear constraints.  
 This allows us to prove that the optimal solutions are in fact attained by discrete spectral measures supported on finitely many atoms.   Second, in the case of balanced 
 portfolia, we establish further structural results for the lower bounds as well as closed form solutions for both the lower- and upper-bounds of extreme Value-at-Risk in the 
 special case of a single extremal coefficient constraint.  The solutions unveil important connections to the Tawn-Molchanov max-stable models.  The results are illustrated
  with two applications: a real data example and closed-form formulae in a market plus sectors framework.\\
  
\end{abstract}

\section{Introduction}

Value-at-Risk (VaR) is one of the predominant risk measures used in
determining minimum capital requirements placed upon financial institutions
in order to cover potential losses in the market. In essence, VaR
is the largest loss having a `reasonable chance' of occurring through
the placement of a risky bet. Formally, if a random variable $X$
represents a loss (negative return) on an asset after a fixed holding
period, and $q\in(0,1)$ is a probability
representing `reasonable chance', we have the following definition
\begin{defn}
The Value-at-Risk of a random variable $X$ at the
level $q\in(0,1),$ denoted $\mathrm{VaR}_{q}(X)$ is defined
as 
\[
\mathrm{VaR}_{q}(X):=\inf\{x\in\mathbb{R}\, :\, \mathbb{P}(X\le x)\ge q\}.
\]
That is, $\mathrm{VaR}_q(X)$ is the (generalized) $100\times q$-th percentile of the loss distribution.
\end{defn}
In practice, financial institutions deal with a multi-dimensional portfolio of statistically
dependent losses $\boldsymbol{X}=(X_{1},X_{2},\ldots,X_{d})^{\top}\in\mathbb{R}^{d}$.
In this case capital requirements should be determined by the value-at-risk
for the sum of losses $\mathrm{VaR}_{q}(S)$, where $S:=X_{1}+ X_{2}+\cdots+ X_{d}$.
In these scenarios it is essential to account for tail dependence
in the components of $\boldsymbol{X}$, see e.g., Embrechts et al \cite{embrechts:lambrigger:wuthrich:2009}. 
Furthermore, regulatory guidelines such as Basel III \cite{baselII} typically
prescribe $q\ge.99$. Hence, the scenario of {\em extreme losses} where $q$ is close
to the value $1$ is of great interest.
Specifically, one is interested in {\em extreme VaR}. Namely, fix a reference asset $X_1$.  Mild multivariate
regular variation conditions on the distribution of $\boldsymbol{X}$, imply the existence of the limit:
\begin{equation}\label{e:extreme-VaR-intro}
 {\cal X} \equiv {\cal X}_{(S,X_1)} := \lim_{q\nearrow 1} {{\rm VaR}_q(S)\over {\rm VaR}_q(X_1) }.
\end{equation}

Following the seminal works of \cite{barbe:fougeres:genest:2006} and \cite{embrechts:lambrigger:wuthrich:2009}, we shall refer to 
the limit ratio ${\cal X}_{(S,X_1)}$ as to {\em extreme VaR}.  It is desirable to be able to bound the extreme VaR coefficient 
${\cal X}$ since it provides the first order approximation of value-at-risk:
$$
{\rm VaR}_q(S) \approx  {\cal X}_{(S,X_1)} \times {\rm VaR}_q(X_1),\ \ \mbox{ for }q\approx 1.
$$

The general goal of this paper is to determine lower- and upper-bounds for extreme VaR under natural constraints
on the portfolio.  This should be contrasted with the statistical problems of estimation of VaR or extreme VaR.  Here, we would like to 
understand and characterize the best- and worst-case scenaria for extreme VaR among all possible models for the joint
(asymptotic) dependence of the losses subject to certain classes of constraints.  In this sense, the type of problem we study
is a constrained and extremal version of the so-called Fr\'echet optimization problems investigated in \cite{ruschendorf:1993} and
recently in \cite{puccetti:ruschendorf:manko:2016,ruschendorf:2017}.

Our motivation stems from potential insolvency in insurance and financial sectors due to catastrophic loss.
In this setting, data on extreme portfolia losses are scarce or non-existent.  Thus, conventional statistical estimation methods are either
difficult to justify or in fact inappropriate for the estimation of extreme VaR.  At the same time, adopting a {\em specific} parsimonious model amounts to imposing (explicitly 
or implicitly) constraints on the asymptotic dependence of the assets.  This can lead to significantly under- or over-estimating the portfolio risk.  Such 
types of challenges motivate us to adopt an alternative perspective of {\em distributionally robust} inference.  That is, we provide upper- and lower-bounds
valid under {\em all possible} extremal dependence scenarios.  Our framework allows the practitioners to incorporate either quantitative constraints on 
easy-to-estimate extremal dependence coefficients or qualitative/structural information such as (partial) extremal independence of the portfolio.

 Value-at-Risk has been studied extensively in the literature. Important theoretical aspects such as the in-coherence of VaR 
 \cite{artzner:delbaen:eber:heath:1999} and its elicitability \cite{ziegel:2016}, for example, are well-understood. At the same time, advanced statistical 
 methodology for the estimation of VaR has been developed accounting for both complex temporal dependence and heavy-tailed marginal distribution of 
 the losses (see e.g., the monograph \cite{mcneil:frey:embrechts:2005}).  Advanced methods for the statistically robust estimation of VaR
 \cite{dupuis:papageorgiou:remillard:2014} exist.  The notion of robust statistical inference should be distinguished from our use of the term 
 distributionally robust inference.  In the former, robustness refers to resilience to outliers in the data within a specified model, in the latter,
 distributionally robust context, the goal is to guard against mis-specifications of the model.  While this perspective has been very popular and
 actively studied in the optimization community (see e.g., \cite{bertsimas:brown:caramanis:2011} and the references therein), only a handful of studies 
 adopt this philosophy in the context of risk measures (see e.g.\ \cite{lam:mottet:2017, engelke:ivanovs:2017, blanchet:chen:zhou:2018, das:dhara:natarjan:2018}). 
 To the best of our knowledge, our work is the first to address the general context of extreme VaR for a multi-dimensional portfolio under extremal coefficient 
 constraints.  To be able to describe our contribution, in the  following Section \ref{sec:intro-notation}, we review some important concepts and notation.
  A summary of our results is given in  Section \ref{sec:results}.
  
 \subsection{Notation and preliminaries}\label{sec:intro-notation}
 
$\bullet$ {\em Regular variation.}  Recall that a random vector $\boldsymbol{X} = (X_i)_{i=1}^d$ is said to be 
multivariate regularly varying (RV), if there exists a non-zero Borel measure $\mu$ on $\mathbb R^d\setminus \{\mathbf 0\}$ 
and a sequence $a_n\nearrow \infty$, such that
\begin{equation}\label{e:d:MRV-intro}
 n\P(a_n^{-1}\boldsymbol{X} \in A) {\longrightarrow} \mu(A),\ \ \mbox{ as }n\to\infty, 
\end{equation}
for all $\mu$-continuity sets $A$ bounded away from the origin. 

The measure $\mu$ in \eqref{e:d:MRV-intro} necessarily satisfies the {\em scaling property}:
\begin{equation}\label{e:d:MRV-intro-scaling}
\mu(c A) = c^{-1/\xi} \mu(A),\ \ \forall c>0,
\end{equation}
for some fixed positive constant $\xi$.  We shall write $\bfX \in RV_{1/\xi}(\{a_n\},\mu)$ and refer to $\xi$ as the index of regular 
variation of the portfolio $\bfX$.
It also follows that the normalization sequence $\{a_{n}\}$ is regularly varying with index $\xi$, i.e., for all $t>0$, we have 
$a_{[t n]}/a_n \to t^{\xi},\ n\to\infty$.  The index $\xi$ does not depend on the choice of the normalization sequence $\{a_n\}$,
and the measure $\mu$ is also essentially unique up to a positive multiplicative constant.   For more details, see 
the Appendix \ref{sec:MRV_EVT} below and the monograph \cite{resnick:2007}.

The scaling relation \eqref{e:d:MRV-intro-scaling} entails that $\mu$ can be conveniently factorized in polar coordinates:
$$
\mu(d\bfx) = r^{-(1+1/\xi)} dr \sigma(d\bfu),
$$
where $r:= \|\bfx\|$ and $\bfu:= \bfx/\|\bfx\|$ are the radial and angular components of 
$\bfx \in\mathbb R^d\setminus \{\bf 0\}$, relative to any (fixed) norm $\|\cdot\|$ in $\mathbb R^d$.
Here, $\sigma$ is a finite positive measure on the unit sphere $\mathbb S:= \{ \bfx \, :\, \|\bfx\|=1\}$,
referred to as a spectral measure of the vector $\bfX$. It is unique up to rescaling by a positive multiplicative factor.

For simplicity, we shall focus here on the case of non-negative losses, i.e., when $\bfX$ takes values in the orthant $[0,\infty)^d$, use the $\ell_1$-norm
$$
 \|\bfx\|:= \sum_{i=1}^d |x_i|,
$$ 
and adopt the following. 

\begin{assumption}\label{cond:RV} Suppose that $\boldsymbol{X} \in RV_{1/\xi}(\{a_n\},\mu)$, where 
the measure $\mu$ is not entirely supported on the hyper-planes $\{ \bfx=(x_i)_{i=1}^d \, :\, x_j=0\},\ j=1,\ldots,d$.
\end{assumption}

This assumption implies that each of the components $X_i,\ i=1,\dots,d$ is heavy-tailed with the same tail index $\xi>0$. 
Indeed, by choosing $A:= sA_i = \{\bfx \in\mathbb R_+^d\, :\, x_i>s\},\ s>0$, in \eqref{e:d:MRV-intro}, and using the scaling  property \eqref{e:d:MRV-intro-scaling}, 
we obtain that for all $s>0$,
\begin{equation}\label{e:Xi-tail}
n \P (X_i > a_n s) \longrightarrow  \vartheta_\bfX(i) s^{-1/\xi},\ \ \mbox{ as }n\to\infty,
\end{equation}
where $\vartheta_\bfX(i):= \mu(A_i)>0 $ is the asymptotic scale coefficient of $X_i$.  
Relation \eqref{e:Xi-tail} implies in particular that the moment $\mathbb E |X_i|^p$ is infinite if $p>1/\xi$ and finite if $0<p<1/\xi$.
The {\em finite-mean} case where $0<\xi<1$ is of primary interest in practice.  Therefore, 
we shall assume throughout that 
$$
 0<\xi \le 1.
$$  
In the infinite-mean case $\xi>1$ an intriguing {\em anti-diversification} phenomenon arises (cf Appendix \ref{sec:xi-trichotomy} below.)

\begin{rem} Assumption \ref{cond:RV} is not very restrictive.  Indeed, it implies that all assets have asymptotically equivalent tails. Had this not 
been the case, only the assets with the heaviest tails would dominate and determine the asymptotic tail behavior of the cumulative loss
$S = X_1+\dots +X_d$.  Thus, when studying extreme VaR, without loss of generality one can focus on the sub-set of losses with heaviest tails.  
\end{rem}

We also standardize the assets to have equal, unit scales such that  \eqref{e:Xi-tail} holds with
\begin{equation}\label{e:Ki-standardization}
 \vartheta_\bfX(i)=1,\ i=1,\ldots,d.
\end{equation}
This standardization does not restrict generality since one can consider the {\em weighted portfolio} 
$$
 S(\bfw) := w_1X_1 +\dots+w_d X_d,
$$
with suitable positive weight vector $\bfw =(w_i)_{i=1}^d$. 

Finally, to separate the roles of the tail behavior and asymptotic dependence, it is convenient to consider the vector
\begin{equation}\label{e:Z}
  \mathbf Z := (X_1^{1/\xi},\, X_2^{1/\xi},\, \cdots,\, X_d^{1/\xi})^\top.
\end{equation}
It can be readily shown that $\bfZ \in RV_{1}(\{b_n\},\nu)$, where $b_n:= a_n^{1/\xi}$ and $\nu(A) := \mu(A^\xi)$.

\medskip
$\bullet$ {\em Extreme VaR formula.} Now,  under the established notation and conditions, Relation \eqref{e:Embrechts} and 
Proposition \ref{p:Barbe-fla-extension} below imply that \eqref{e:extreme-VaR-intro} holds. That is, 
extreme VaR is well-defined, and it has, moreover, the following closed-form expression:
\begin{equation}\label{e:rho_w}
{\cal X}_{(S(\bfw),X_1)} =   \rho_{\bfw} ^{\xi}
\quad \mbox{ with }\quad
 \rho_{\bfw} \equiv \rho_\bfw(H,\xi)  :=  \int_{\mathbb S_+} {\Big(} w_1 u_1^\xi +\cdots + w_d u_d^{\xi} {\Big)}^{1/\xi} H(d\bfu),
\end{equation}
where  $ \mathbb S_+ = \{ \bfx \ge \mathbf 0\, :\, \|\bfx\|=1\}$ is the unit simplex in $\mathbb R^d$.

Here $H$ is the (unique) spectral measure of the vector $\bfZ$ satisfying the marginal moment constraints 
\begin{equation}
 1 =\int_{\mathbb{S}_{+}}u_{j}H (d\boldsymbol{u}),\ j=1,\ldots,d.\label{eq:marginal_constraints}
\end{equation}
Note that since $\sum_{j=1}^d u_j = \|\bfu\| = 1,\ \bfu\in\mathbb S_+$ we have $H(\mathbb S_+)=d.$

Well-known {\em Hoeffding-Fr\'echet type} universal bounds on the value of $\rho_{\bfw}\equiv \rho_{\bfw} (H,\xi)$ are given by
\begin{equation}
\sum_{i=1}^d w^{1/\xi}_i  \le  \rho_\bfw(H,\xi)\le {\Big(}\sum_{i=1}^d w_i {\Big)}^{1/\xi}\quad (0< \xi\le1)\label{eq:universal_rhobds_le1}
\end{equation}
(see e.g. Corollary 4.2 in Embrechts et al \cite{embrechts:lambrigger:wuthrich:2009}).  These
inequalities follow readily from \eqref{e:rho_w}.

The lower bound $\rho_{\bfw} = \sum_{i=1}^d w_i^{1/\xi}$ in \eqref{eq:universal_rhobds_le1} corresponds to (asymptotic) independence 
and the upper bound $\rho_\bfw = {\Big(}\sum_{i=1}^d w_i {\Big)}^{1/\xi}$ to complete tail dependence, where all
components of the vector $\boldsymbol{X}$ are asymptotically identical. This agrees with our intuition about {\em diversification}, where holding 
independent assets leads to the lowest value of extreme VaR, while complete dependence corresponds to the worst case of risk. 
Surprisingly, this intuition is reversed in the infinite-mean regime $\xi>1$ (see Appendix \ref{sec:xi-trichotomy} below.)

\medskip
$\bullet$ {\em Extremal coefficients.} The Hoeffding--Fr\'echet type bounds in \eqref{eq:universal_rhobds_le1} are rather wide.  In practice, however, the range of possible 
values $\rho_\bfw$ can be  significantly reduced under suitable constraints on the extremal dependence of the portfolio.  
In this work, we focus on so-called {\em extremal coefficient} constraints, which capture (in a rough sense) the strength of
tail dependence amongst a given \emph{subset} of assets in the portfolio $\boldsymbol{X}$.

Specifically, for any non-empty set of assets $J \subset \{1,\dots,d\}$ by taking $A:= sA_J = \{\bfx \in\mathbb R_+^d\, :\, x_j>s,\ \mbox{ for some }j\in J\},\ s>0$, 
in Relation \eqref{e:d:MRV-intro}, we obtain
$$
n \P ( \max_{j\in J} X_j > a_n s)  \longrightarrow  \vartheta_\bfX(J) s^{-1/\xi},\ \ \mbox{ as }n\to\infty,
$$
where
$
\vartheta_\bfX(J):=  \mu(A_J)>0
$
is now the asymptotic scale coefficient of the maximum loss $\max_{j\in J} X_j$ over $J$.  The coefficients $\vartheta_\bfX(J)$ will be referred to as
{\em extremal coefficients} of the portfolio $\bfX$.  By Lemma \ref{l:rho} below 
\begin{equation}\label{e:ext-coeffs}
 \vartheta_\bfX(J) = \int_{\mathbb S_+} \max_{j\in J} u_j H(d\bfu),
\end{equation}
where $H$ is the {\em same} spectral measure appearing in \eqref{e:rho_w}.  This, in view of  \eqref{eq:marginal_constraints}, readily implies
\begin{equation}\label{e:intro-theta-bounds}
 \max_{j\in J} \int_{\mathbb S_+} H(d\bfu) = 1 \le  \vartheta_\bfX(J) \le  |J| = \sum_{j\in J} \int_{\mathbb S_+} u_j H(d\bfu),
\end{equation}
where $|J|$ is the size of the set $J$. The upper bound is attained when the $X_j$'s are {\em asymptotically independent}, while the lower bound
corresponds to  the case of perfect asymptotic dependence, e.g., $X_{j_1}=\cdots=X_{j_\ell}$, for $J = \{j_1,\cdots,j_\ell\}$.

The extremal coefficients naturally encode a great variety (although not all) extremal dependence relationships among the assets.  For example, the 
classic upper tail dependence coefficient is expressed as follows
$$
\lambda_\bfX (\{i,j\}) := \lim_{q\uparrow 1} \P( F_{X_i}(X_i) > q | F_{X_j}(X_j)>q) = 2 - \vartheta_{\bfX}(\{i,j\}),
$$
for all $1\le i\not=j\le d$, where $F_{X}(x) = \P(X\le x)$ denotes the cumulative distribution function of a random variable $X$.
In this case, the bounds \eqref{e:intro-theta-bounds} amount to $0\le \lambda_\bfX(\{i,j\})\le 1$,
where $\lambda_\bfX(\{i,j\})=0$ corresponds to asymptotic independence and $\lambda_\bfX(\{i,j\})=1$ to perfect asymptotic dependence.

As another example, the $d$-variate extremal coefficient 
\begin{equation}\label{e:single-d-variate-intro}
\vartheta_d\equiv {\vartheta}(\{1,\dots,d\}) = \lim_{u\to \infty } {\P( \max_{i=1,\dots,d} X_i >u) \over \P(X_1>u)}
\end{equation}
takes values in the range $[1,d]$.  It quantifies the degree to which {\em all assets} in the portfolio experience an extreme loss {\em simultaneously}.  
For example, $\vartheta_d$ equals $1$ under perfect asymptotic dependence (e.g., $X_1=\cdots=X_d$) and it equals $d$ if the assets are 
{\em asymptotically} independent with equal scales. 

\subsection{Summary of our contributions} \label{sec:results}

In view of \eqref{e:rho_w} determining the best- and worst-case extreme VaR scenaria amounts to solving a pair of 
infinite-dimensional optimization problems over a space of admissible spectral measures ${\cal H}$. 
Namely, we consider a {\em large} family ${\cal H}$ of spectral measures and posit the optimization problems
\begin{eqnarray}\label{e:ell-u-problems}
\L({\cal H}):= \inf_{H\in {\cal H}} \rho_\bfw(H,\xi)\quad \mbox{ and }  \quad \U({\cal H}):= \sup_{H\in{\cal H}} \rho_\bfw(H,\xi),
\end{eqnarray}
 where 
$
 \rho_\bfw(H,\xi) = \int_{\mathbb S_+} (w_1 u_1^\xi +\cdots + w_d u_d^\xi )^{1/\xi} H(d\bfu).
$

Then, in view of \eqref{e:rho_w}, we obtain the following {\em universal} lower and upper bounds for extreme VaR:
\begin{equation}\label{e:ext_VaR-bounds}
 \L^\xi({\cal H}) \le {\cal X}_{(S,X_1)} \le \U^\xi({\cal H}).
\end{equation}
If the class ${\cal H}$ includes all admissible (normalized) spectral measures, these bounds can be rather wide 
(see Relation \ref{eq:universal_rhobds_le1}), which may limit their practical value in establishing capital requirements.  
As indicated, we consider classes of {\em all possible} spectral measures ${\cal H}$ that satisfy extremal coefficient 
constraints such that
$$
\vartheta_\bfX(J) \equiv \int_{\mathbb S_+} \max_{j\in J} u_j H(d\bfu) = c_J,\ \ J\in {\cal J},
$$
for a given family of non-empty subset of assets ${\cal J}\subset 2^{\{1,\cdots,d\}}\setminus\emptyset$.  The standardization 
\eqref{e:Ki-standardization} corresponds to the singleton sets $J=\{i\}$ and $c_{\{i\}}=1$, which for our applications, will be always included as constraints.  

The constants $c_J$ can be either estimated or assigned by a domain expert.  They can be used to encode structural information such as
asymptotic independence (cf Section \ref{sec:mkt-sec}, below). Figure \ref{fig:d-variate} illustrates that the knowledge of the single $d$-variate extremal coefficient 
constraint in \eqref{e:single-d-variate-intro} can dramatically reduce the range of {\em all possible} extreme VaR ${\cal X}$, even in dimensions as high as $d=100$.  
\begin{figure}[H]
\centering{}\includegraphics[width=6in]{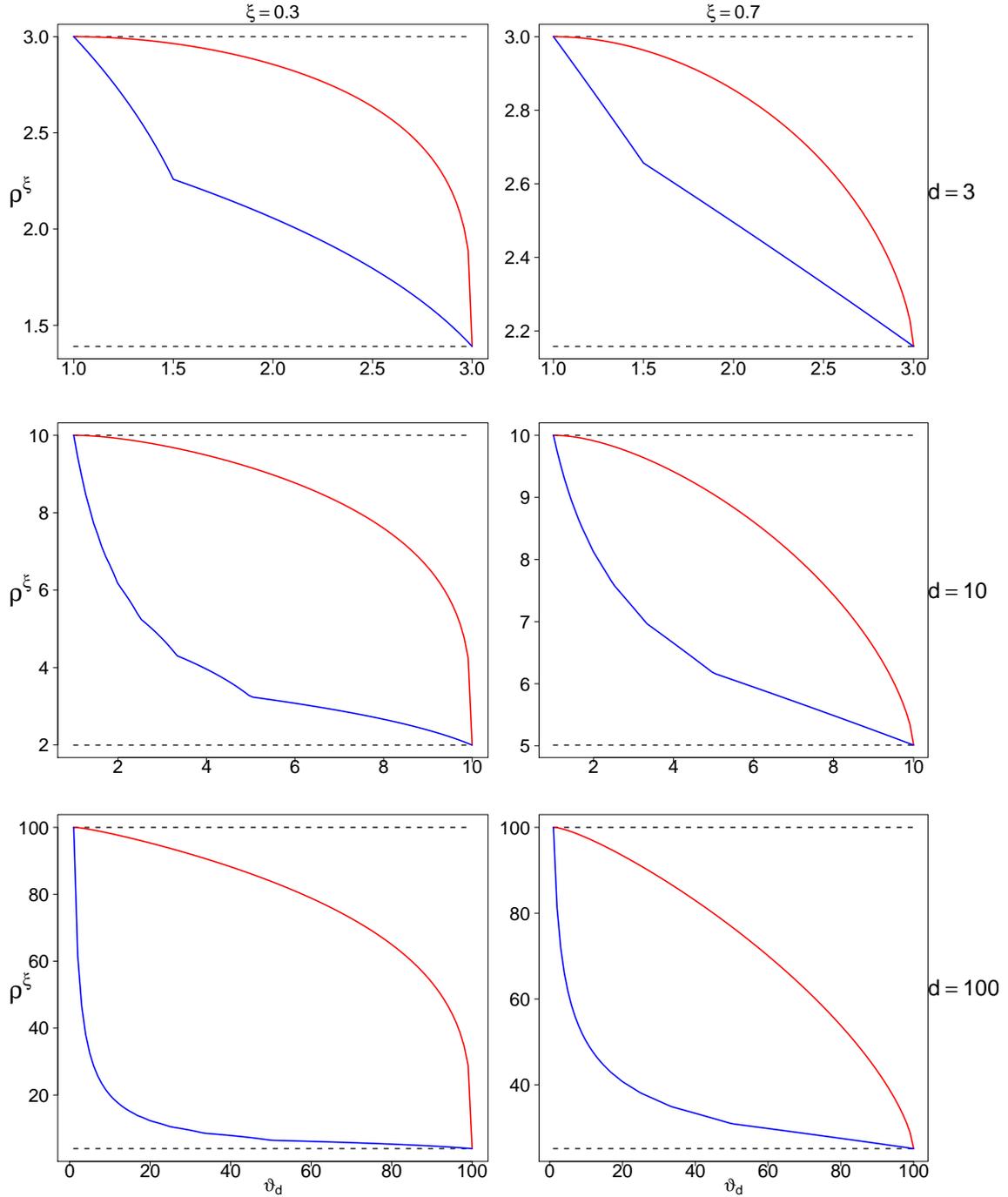}
\protect\caption{Upper and lower bounds on extreme VaR $\rho^{\xi}$ when given a single
fixed $d$-variate extremal coefficient $\vartheta_{d}$ constraint.\label{fig:d-variate} }
\end{figure}

\medskip
{\bf Contributions.} Observe that both the objective function in \eqref{e:rho_w} and the constraints in \eqref{e:ext-coeffs} 
are {\em linear} in the parameter $H$.  The challenge is, however, that $H$ takes values in an infinite-dimensional space of measures.  
Our findings can be summarized by three main themes:

$\bullet$ {\em Optimal measures have finite support.} We establish structural results showing that the infimum and 
supremum of $\rho_\bfw$ are attained by {\em discrete measures} that are supported 
on a finite set of atoms. In each case, the number of atoms is not more than 
the number of constraints (Theorem \ref{thm:main_redux}).  Thus, in principle, the linear 
infinite-dimensional problems reduce to non-linear finite-dimensional 
optimization problems.  These results stem from a fundamental 
connection with the theory of linear semi-infinite optimization 
outlined in Section \ref{sec:LSIP} below. 


$\bullet$ {\em A Tawn-Molchanov minimizer and a convex maximizer.} Surprisingly, the infimum of $\rho_{\bfw}$ and in turn the lower bound 
on ${\cal X}$ is attained by measures with the {\em same support} as the celebrated Tawn-Molchanov models in Strokorb  
and Schlather  \cite{strokorb2015}.  
This allows us to further reduce the optimization to a linear program, which can be solved {\em exactly} using conventional linear solvers
in moderate dimension. We also establish that the maximization problem reduces to an ordinary convex optimization problem which can be solved 
in polynomial time within arbitrary precision.  Efficient solvers for these optimization problems have yet to be implemented, 
nevertheless our theoretical results suggest that they can be efficiently solved.


$\bullet$ {\em Closed form solutions.} Finally, in the case of a single $d$-variate constraint, we establish {\em closed form} expressions 
for both the lower- and upper-bounds, which are valid {\em in arbitrary} dimensions.  These formulae were 
used in Figure \ref{fig:d-variate} and further leveraged in Section \ref{sec:mkt-sec} to illustrate how conditional independence can lead to 
very substantial reduction of the range of extreme VaR.

\medskip
{\em The rest of the paper is structured as follows.} Section \ref{sec:background} reviews key results from the theory of linear semi-infinite programming (LSIP)
and demonstrates that our optimization problem can be viewed as a dual of a LSIP.  This connection is further explored in Section \ref{sec:Main-results}, where
the main results on the general characterization of the spectral measures attaining the minimum and maximum extreme VaR are presented.
Section \ref{sec:closed_form} proceeds with more detailed results on in the cases of the Tawn-Molchanov minimizer and our closed form solutions.  
Section \ref{sec:applications} briefly illustrates the established theory.  In Section \ref{sec:industry-portfolia}, using a data set of industry portfolia, we show how utilizing all
bi-variate extremal coefficient constraints can lead to tight bounds on extreme VaR, which are in close agreement with semi-parametric
estimates obtained using Extreme Value Theory.  In Section \ref{sec:mkt-sec}, we provide a practical application of the closed-form formulae for the bounds on 
extreme VaR in a context of a market and sectors model.  This demonstrates, how expert knowledge on the structure of the market can be encoded
via extremal coefficient constraints in cases where data may be scarce. The proofs and auxiliary facts from optimization are collected in the Appendix.

\section{A Connection to Linear Semi-Infinite Programming} \label{sec:background}

In this section, we will show that our optimization problems are in fact duals to 
{\em linear semi-infinite programming} (LSIP) problems.  This will lead to profound
structural results and certain closed-form solutions for upper and lower bounds on extreme VaR.

\subsection{Problem formulation}

Recall that we want to solve the pair of optimization
problems:
\begin{eqnarray}
(\L) \quad\quad\quad &  & \inf_{H}\rho_\bfw(H,\xi) \label{e:L-bound}\\
(\U) \quad\quad\quad & &  \sup_{H}\rho_\bfw(H,\xi)\label{e:U-bound}\\
\text{subject to:} &  & \int_{\mathbb{S}_{+}}\max_{j\in J}\{u_{j}\}H(d\boldsymbol{u})=c_{J},\text{ for all }J\in\mathcal{J},\label{eq:H_constraints}
\end{eqnarray}
where ${\cal J} \subset 2^{\{1,\cdots,d\}}$,
is a collection of non-empty subsets of indices $\{1,\cdots,d\}$;
the functional $\rho_\bfw$ is in \eqref{e:rho_w}; and the supremum and infimum are taken over all finite measures $H$
on $\mathbb{S}_{+}$ that satisfy the extremal coefficient constraints in \eqref{eq:H_constraints}. 

\begin{rem} \label{rem:consistent-thetas} A set of non-negative constants $\boldsymbol{c}=(c_{J})_{J\subset\{1,\ldots,d\}}\in\mathbb{R}_{+}^{2^{d}-1}$ can be
the extremal coefficients of a random vector $\bfX$, if and only if they satisfy the {\em consistency relationships}
\[
\boldsymbol{c}\in\varTheta:=\left\{ \boldsymbol{\vartheta}\in\mathbb{R}_{+}^{2^{d}-1}:\sum_{L\,  :\, J \subseteq L}(-1)^{|L\setminus J|+1}\vartheta(L)\ge0,\ \text{for all }J\subsetneq\{1,\ldots,d\}\right\}.
\]
See Corollary 5 in \cite{schlather:tawn:2002} and \cite{strokorb2015} for more details.
\end{rem}

Extremal coefficients are only summary, moment-type functionals, and they alone do not fully characterize the spectral
measure $H$, except in special cases \cite{strokorb2015}. In general, however, it is not known to what 
extent the full or partial knowledge of the extremal coefficients confine the set of possible values of $\rho_\bfw$ and hence extreme VaR. 
This is one of the motivations for our work. 

\begin{assumption}
\label{as:margin_const}We assume that the marginal constraints \eqref{eq:marginal_constraints}
are always included in \eqref{eq:H_constraints} by requiring that
the singletons $\{1\},\ldots,\{d\}$ belong to $\mathcal{J}$ and
$c_{\{j\}}=1$ for $j=1,\ldots,d.$ To avoid further situations that result in trivial optimization problems, we
also assume $\mathcal{J}$ is sufficiently rich such that 
\[
1=\sum_{j=1}^{d}u_{j}<\sum_{J\in\mathcal{J}}\max_{j\in J}\{u_{j}\},\ \text{for all }\boldsymbol{u}\in\mathbb{S}_{+}.
\]
In particular, this holds if $\mathcal{J}$ includes all pairs or the set
$\{1,\ldots,d\}\in\mathcal{J}$.
\end{assumption}

\subsection{Linear semi-infinite programming\label{sec:LSIP}}

The purpose of this section is to review definitions and notations
from the field of {\em linear semi-infinite programming} (LSIP) that we will use throughout
this paper (see also Appendix \ref{subsec:KKT}).  Our main contributions
in the following Section \ref{sec:Main-results} such as the existence of solutions to 
$(\L)$ and $(\U)$ with finite support (reducibility) and
and exact formulae for the optimum will leverage powerful results from this established theory. 
Those interested in a more comprehensive treatment is referred to the monograph of Goberna and Lopez \cite{goberna:lopez:1998} as well
as the review by Shapiro \cite{shapiro:2009}.  See also \cite{goberna:lopez:2018} for a survey of recent
advancements in LSIP.

\medskip
\noindent  {\bf Formulation.} Linear semi-infinite programs are formulated as follows: 
\begin{eqnarray*}
(P)\quad\quad\quad\quad\inf_{\boldsymbol{x}\in\mathbb{R}^{p}} &  & \boldsymbol{c}^{\top}\boldsymbol{x}\\
\text{subject to:} &  & b(t)-\boldsymbol{a}(t)^{\top}\boldsymbol{x}\le0,\ t\in T,
\end{eqnarray*}
where $T$ is a (possibly infinite) index set. For a given mathematical
program, say $(\tilde{P})$, we use the notation $\text{val}(\tilde{P}),$
to denote its optimal value while $\text{sol}(\tilde{P})$ denotes
the solution set, i.e. the set of feasible points that yield optimal
values. Generally, $\text{val}(\tilde{P})$ may be infinite and $\text{sol}(\tilde{P})$
my be empty. If $\text{sol}(\tilde{P})=\varnothing$, then by convention
$\text{val}(\tilde{P})=\infty$ and we say $(\tilde{P})$ is \emph{unsolvable.}

The following assumption establishing the \emph{continuity} of $(P)$ (in the language of LSIPs) has far reaching 
consequences in terms of the structure of solutions to $(P)$. 

\begin{assumption}
\label{as:main} In $(P)$, we suppose $T$ is a compact subset of $\mathbb{R}^{d}$ and
$\boldsymbol{a}:T\mapsto\mathbb{R}^{p}$, $b:T\mapsto\mathbb{R}$ are continuous and hence bounded on $T$. 
\end{assumption}
Thus, we define the \emph{Lagrangian} of problem $(P)$ as the function

\emph{$L:\mathbb{R}^{p}\times\Omega\mapsto\mathbb{R}$} \emph{
\begin{equation}
L(\boldsymbol{x},\omega)=\boldsymbol{c}^{\top}\boldsymbol{x}+\int_{T}\left(b(t)-\boldsymbol{a}(t)^{\top}\boldsymbol{x}\right)\omega(dt),\label{eq:lagrangian}
\end{equation}
}where $\Omega$ is the space of finite (non-negative) Borel measures
on $T$. 

\begin{rem} Assumption \ref{as:main} allows us to express the Lagrangian function as
\eqref{eq:lagrangian}. This follows from the fact that the topological
\emph{dual space} of continuous functions on the compact set $T\subset\mathbb{R}^{p}$
is indeed the space of Borel measures on $T$.  For more details
see e.g.\ Ch.\ 2 of \cite{goberna:lopez:1998}. 
\end{rem}

\begin{rem} While Assumption \ref{as:main} appears as a rather strong condition in the literature of LSIP,
 we will show in Section \ref{sec:Main-results} that Assumption \ref{as:main} is naturally satisfied for our main
 motivating problems $(\U)$ and $(\L)$.
\end{rem}

\noindent {\bf Duality.} We define the \emph{dual function} $g:\Omega\mapsto\mathbb{R}$ as
\[
g(\omega)=\inf_{\boldsymbol{x}\in\mathbb{R}^{p}}L(\boldsymbol{x},\omega).
\]
The dual function yields a lower bound on the optimal value of $(P)$. Indeed, by $(P)$,
for any feasible\emph{ }$\tilde{\boldsymbol{x}}\in\mathbb{R}^{p}$, it follows that 
\[
\int_{T}\left(b(t)-\boldsymbol{a}(t)^{\top}\tilde{\boldsymbol{x}}\right)\omega(dt)\le0,
\]
which implies
\begin{equation}
g(\omega)=\inf_{\boldsymbol{x}\in\mathbb{R}^{p}}L(\boldsymbol{x},\omega)\le\boldsymbol{c}^{\top}\tilde{\boldsymbol{x}}+\int_{T}\left(b(t)-\boldsymbol{a}(t)^{\top}\tilde{\boldsymbol{x}}\right)\omega(dt)\le\boldsymbol{c}^{\top}\tilde{\boldsymbol{x}}.\label{eq:dual_lwrbound}
\end{equation}
 The fact that the feasible $\tilde{\boldsymbol{x}}$ was arbitrary implies $g(\omega)\le\mathrm{val}(P).$ 
 This inequality is trivial unless $\int_{T}\boldsymbol{a}(t)\omega(dt)=\boldsymbol{c}$. Indeed, otherwise if 
 $\int_{T}\boldsymbol{a}(t)\omega(dt)\not =\boldsymbol{c}$, then for some $\boldsymbol{x_0},$ we have
 $\mathbf c^\top\boldsymbol{x_0} - \int_T \boldsymbol{a}(t) ^\top \boldsymbol{x_0} \omega(dt) <0$, and hence by Assumption \ref{as:main}, it follows that 
 $g(\omega)=\inf_{\boldsymbol{x}\in\mathbb{R}^{p}}L(\boldsymbol{x},\omega) = -\infty$.
 
 Therefore, only measures $\omega \in \Omega$ for which $\int_{T}\boldsymbol{a}(t)\omega(dt)=\boldsymbol{c}$ holds
 are of interest and they are referred to as \emph{dual feasible}. Thus we arrive at the following
\emph{dual problem}:
\begin{eqnarray*}
(D)\quad\quad\quad\quad\sup_{\omega\in\Omega} &  & \int_{T}b(t)\omega(dt)\\
\text{subject to:} &  & \int_{T}\boldsymbol{a}(t)\omega(dt)=\boldsymbol{c}.
\end{eqnarray*}

In view of \eqref{eq:dual_lwrbound}, we have that 
\begin{equation}\label{e:dual-bound}
 \sup_{\omega\in \Omega} \inf_{\boldsymbol{x}\in \mathbb R^d} L(\boldsymbol{x}, \omega) = {\rm val}(D) \le {\rm val}(P).
\end{equation}
A common task with many optimization problems is to determine the
existence (or non-existence) of a \emph{duality gap, $|\mathrm{val}(P)-\mathrm{val}(D)|$.
}If $\mathrm{val}(P)=\mathrm{val}(D)$, then it suffices to
solve either $(P)$ or $(D)$ to obtain the optimal value, so long
as both problems are \emph{solvable}. The condition $\mathrm{val}(P)=\mathrm{val}(D)$
with $\mathrm{sol}(D)\not=\varnothing$ is known as \emph{strong duality.}
If $(P)$ is \emph{solvable}, i.e. $\mathrm{val}(P) < \infty$, then under assumption \ref{as:main}, a sufficient condition for strong duality of $(P,D)$ is \emph{Slater's Condition}, i.e. there
exists $\tilde{\boldsymbol{x}}\in\mathbb{R}^{p}$ such that 
\begin{equation}\label{e:Slater-cond}
b(t)-\boldsymbol{a}(t)^{\top}\tilde{\boldsymbol{x}}<0,\ \text{for all }t\in T.
\end{equation}
See Theorem 2.3 in \cite{shapiro:2009} for further details on Slater's condition and strong duality for LSIPs.

The above discussion reveals a fundamental connection between the two optimization problems in \eqref{e:L-bound} and 
\eqref{e:U-bound} and the theory of LSIP.

\begin{cor} The problem of finding the {\em upper bound} $(\U)$ in \eqref{e:U-bound} under extremal coefficient constraints \eqref{eq:H_constraints} 
is the dual to an LSIP problem $(P)$, where 
$$
T = \mathbb S_+,\quad \mathbf a_J(t) = \max_{j\in J}\, t_j,\ J\in {\cal J},\quad b(t) = \left(\sum_{i=1}^d w_i t_i^\xi\right)^{1/\xi}\quad\mbox{ and }\quad \mathbf c = (c_J)_{J\in {\cal J}}.
$$
Similarly, the problem of finding the lower bound $(\L)$ in \eqref{e:L-bound} is the dual of an LSIP involving maximization, where
formally `$\sup$' is reduced to `$\inf$' by changing the sign of the objective function.
\end{cor}

This connection allows us to employ powerful results from the LSIP theory discussed next.

\medskip
\noindent {\bf Reducibility.} The following discussion lays the groundwork for establishing the finite support of optimal solutions to $(\L)$ and $(\U)$. 
Consider a finite index set $T_{m}\subset T$ with $|T_{m}|\le m.$
Solving problem $(P)$ when the constraints are restricted to the
finite set $T_{m}$ reduces to a {\em standard linear program} 
\begin{eqnarray*}
(P_{m})\quad\quad\quad\quad\inf_{\boldsymbol{x}\in\mathbb{R}^{p}} &  & \boldsymbol{c}^{\top}\boldsymbol{x}\\
\text{subject to:} &  & b(t_{i})-\boldsymbol{a}(t_{i})^{\top}\boldsymbol{x}\le0,\ t_{i}=T_{m},\ i=1,\ldots,m,
\end{eqnarray*}
which yields the corresponding dual 
\begin{eqnarray*}
(D_{m})\quad\quad\quad\quad\sup_{\boldsymbol{\omega}\in\mathbb{R}_{+}^{m}} &  & \sum_{i=1}^{m}b(t_{i})\omega_{i}\\
\text{subject to:} &  & \sum_{i=1}^{m}\boldsymbol{a}(t_{i})\omega_{i}=\boldsymbol{c},\ t_{i}=T_{m},\ i=1,\ldots,m.
\end{eqnarray*}
Problem $(P_{m})$ is called a \emph{discretization }of $(P)$. The
feasible set for $(P)$ is contained in the feasible set for $(P_{m})$.
Hence, $\mathrm{val}(P_m)\le\mathrm{val}(P).$ If for every $\varepsilon>0$,
there exists $(P_{m(\varepsilon)})$ such that $\mathrm{val}(P)-\mathrm{val}(P_{m(\varepsilon)})\le\varepsilon$
than we say $(P)$ is \emph{discretizable. } Whereas, if there exists $(P_{m})$
such that $\mathrm{val}(P_{m})=\mathrm{val}(P)$ then $(P)$ is said
to be \emph{reducible.}  In this case, on the language of measures, the optimum is attained by a discrete measure 
$\omega(dt) = \sum_{i=1}^m \nu_i \delta_{\{t_i\}}(dt)$ with a finite support $\{t_1,\ldots,t_m\}\subset T$.

\begin{rem} Even if an LSIP is theoretically reducible, it may be challenging to find the actual support
set of an $\omega \in {\rm sol}(D)$. This is because finding the support amounts to solving a non-linear optimization problem.
\end{rem}

The following proposition establishes conditions for the \emph{reducibility} of the LSIP $(P)$.
\begin{prop}[Theorem 3.2 in \cite{shapiro:2009}]
\label{prop:LSIP_reducibility} Suppose that for problem $(P)$, Assumption \ref{as:main} holds and $\mathrm{val}(P)<\infty$.
If for any $\{t_{1},t_{2},\ldots t_{p+1}\}\subset T$, there exists $\boldsymbol{x}\in\mathbb{R}^{p}$
such that 
\begin{equation}\label{e:slater_disc_m}
\boldsymbol{a}(t_{k})^{\top}\boldsymbol{x}>b(t_{k}),\ k=1,\ldots,p+1.
\end{equation}
Then there exists $\{t_{1},\ldots,t_{m}\}=T_{m}\subset T$ with $m\le p$
such that for corresponding discretizations $(P_{m})$ and $(D_{m})$
\[
\mathrm{val}(P)=\mathrm{val}(P_{m})=\mathrm{val}(D_{m})=\mathrm{val}(D).
\]
\end{prop}

Note that if Slater's condition holds for $(P)$, then \eqref{e:slater_disc_m} is satisfied.  Which yields the
following corollary

\begin{cor}\label{cor:lsip_sol_p_atoms}
If Assumption \ref{as:main} holds for $(P)$, $\mathrm{val}(P)<\infty$ and Slater's condition holds, 
then there exists a (strong) dual pair $(P,D)$ and $\omega\in\mathrm{sol}(D)\subset\Omega$
such that $\omega$ is finitely supported on at most $p$ atoms $\{t_{1},t_{2},\ldots,t_{p}\}\subset T$. 
\end{cor}

\section{Main results}\label{sec:Main-results}

\subsection{Optimal measures with finite support}

In this section, we establish general structural results for problems $(\L)$ and $(\U)$ by exploiting their duality to 
linear semi-infinite programs (LSIPs) discussed above.  We show that the optimum are attained by measures with finite support and 
we prove that $(\U)$ is equivalent to a finite dimensional convex optimization problem, which can be solved
in polynomial time.

\comment{ Recall that 
$\mathrm{val}(\cdot)$ and $\mathrm{sol}(\cdot)$ denote the optimal
value and solution set of an optimization problem, respectively.}

\begin{thm}
\label{thm:LSIPexists} If Assumption \ref{as:margin_const}
holds, then there exist (primal) linear semi-infinite programs $(\L^{\prime})$ and $(\U^{\prime})$, whose
dual problems are $(\L)$ and $(\U)$, respectively.  Furthermore, for $(\L^{\prime})$ and $(\U^{\prime})$, we have:


{\rm (i)} Assumption \ref{as:main} is satisfied.

{\rm (ii)} The Slater condition holds.

{\rm (iii)} The optimal values are finite.

{\rm (iv)} Strong duality holds for the pairs $(\L, \L^{\prime})$ and $(\U,\U^{\prime})$.

{\rm (v)} The problems are reducible.

{\rm (vi)} There exists solutions to $(\L)$ and $(\U)$ that are supported on at most $|\mathcal{J}|$ atoms.
\end{thm}

\begin{proof} We consider only problems $(\U)$ and $(\U^\prime)$. The arguments for $(\L)$ and $(\L^{\prime})$ are similar.

Let $p=|\mathcal{J}|$ and $\boldsymbol{c}=(c_{J})_{J\in\mathcal{J}}\in\mathbb{R}_{+}^{p}.$
Define the continuous functions $b:\mathbb{S}_{+}\mapsto\mathbb{R}_{+}$,
$\boldsymbol{a}:\mathbb{S}_{+}\mapsto\mathbb{R}_{+}^{p}$ 
\[
b(\boldsymbol{u})=\left(w_1 u_{1}^{\xi}+\cdots+w_d u_{d}^{\xi}\right)^{1/\xi}
\quad \mbox{ and }\quad
\boldsymbol{a}(\boldsymbol{u})=\left(\max_{j\in J}\{u_{j}\}\right)_{J\in\mathcal{J}}.
\]
Consider the linear semi-infinite program 
\begin{eqnarray} \label{e:U_rho-prime}
(\U^{\prime})\quad\quad\quad\quad\inf_{\boldsymbol{x}\in\mathbb{R}^{p}} &  & \boldsymbol{c}^{\top}\boldsymbol{x} \\
\text{subject to:} &  & b(\boldsymbol{u})-\boldsymbol{a}(\boldsymbol{u})^{\top}\boldsymbol{x}\le0,\ \boldsymbol{u}\in\mathbb{S}_{+}.  \nonumber
\end{eqnarray}

Letting $\mathcal{H}$ denote the space of finite Borel measures on $\mathbb{S}_{+}$, by the Lagrangian
duality theory discussed in Section \ref{sec:LSIP}, it follows that the dual of $(\U^{\prime})$ is 
\begin{eqnarray*}
\quad\quad\quad\quad\sup_{H\in\mathcal{H}} &  & \int_{\mathbb{S}_{+}}\left(w_1 u_{1}^{\xi}+\cdots+w_d u_{d}^{\xi}\right)^{1/\xi}H(d\boldsymbol{u})\\
\text{subject to:} &  & \left\{ \int_{\mathbb{S}_{+}}\max_{j\in J}\{u_{j}\}H(d\boldsymbol{u})=c_{J}\right\} _{J\in\mathcal{J}},
\end{eqnarray*}
which is in fact problem $(\U)$ in \eqref{e:U-bound}. This establishes the desired duality of $(\U)$ to the above LSIP $(\U^\prime)$.

Now, observe that $(\U^\prime)$ satisfies Assumption \ref{as:main}, since $\mathbb{S}_{+}\subset\mathbb{R}^{d}$ is compact and
the functions $b$ and $\boldsymbol{a}$ are continuous on $\mathbb{S}_{+}$. This proves {\em (i)}.

We next show {\em (ii)}. Observe that for all $\boldsymbol{u}\in\mathbb{S}_{+},$ we have
\begin{eqnarray}
b(\boldsymbol{u})&=&\left(w_1 u_{1}^{\xi}+\cdots+w_d u_{d}^{\xi}\right)^{1/\xi}
\le {\Big(} \max_{j=1,\dots,d} w_j^{1/\xi} {\Big)} \sum_{j=1}^{d}u_{j}\nonumber\\ 
&=:& C_\bfw \sum_{j=1}^{d}u_{j} < C_\bfw \boldsymbol{a}(\boldsymbol{u})^{\top}\boldsymbol{1}, 
 \label{eq:1_is_feasible}
\end{eqnarray}
where inequality in \eqref{eq:1_is_feasible} follows from
Assumption \ref{as:margin_const}.  Hence $\tilde{\boldsymbol{x}}\equiv C_\bfw \mathbf{1} \in\mathbb{R}^{p}$
is primal feasible for the LSIP program $(\U^\prime)$ and the Slater condition \eqref{e:Slater-cond} holds.

In view of \eqref{eq:universal_rhobds_le1}, \eqref{e:dual-bound}, and \eqref{eq:1_is_feasible}, 
we obtain
\[
-\infty<{\Big(} \sum_{i=1}^d w_i {\Big)}^{1/\xi} \le \mathrm{val}(\U)
                                                                       \le\mathrm{val}(\U^{\prime})
                                                                       \le C_\bfw\sum_{J\in\mathcal{J}}c_{J}<\infty,
\]
which proves {\em (iii)}.

Finally, by Proposition \ref{prop:LSIP_reducibility} (c.f. Corollary \ref{cor:lsip_sol_p_atoms}), {\em (i)}, {\em (ii)}, and {\em (iii)} are sufficient for {\em (iv)}, {\em (v)}, and {\em (vi)}.
\end{proof}

The fact that Theorem \ref{thm:LSIPexists}\emph{(vi)} implies that the optimal values of $(\L)$ and $(\U)$ can be achieved by measures concentrated on 
at most $|{\cal J}|$ atoms leads to the following characterization of ${\rm val}(\L)$ and ${\rm val}(\U)$.

\begin{thm}
\label{thm:main_redux} Recall the extremal coefficient constraints $\boldsymbol{c} = (c_J)_{J\in\mathcal{J}} \in \mathbb{R}^{|\mathcal{J}|}_+$ in \eqref{eq:H_constraints}
for problems $(\L)$ and $(\U)$. Define the set of non-negative $d\times |{\cal J}|$ matrices
\begin{equation}
\mathcal{A}_{\boldsymbol{c}}:=\left\{ A\in\mathbb{R}_{+}^{d \times |\mathcal{J}|}\, :\, \sum_{K\in\mathcal{J}}\max_{j\in J}\{a_{jK}\}=c_{J},J\in\mathcal{J}\right\} .\label{eq:maxlin_const}
\end{equation}
Then, by letting $f(A):= \sum_{K\in {\cal J}} \left(w_1 a_{1K}^{\xi}+\cdots+w_d a_{dK}^{\xi}\right)^{1/\xi}$, we have
\[
{\rm val }(\L) = \inf_{A\in\mathcal{A}_{\boldsymbol{c}}} f(A) \quad\mbox{ and }\quad {\rm val}(\U)=
\sup_{A\in\mathcal{A}_{\boldsymbol{c}}} f(A).
\]
\end{thm}
\begin{proof}
Theorem \ref{thm:LSIPexists}\emph{(vi)} implies that there exists 
a discretization $(\L^{m})$ with $m\le|\mathcal{J}|$ such that $\mathrm{val}(\L^{m})=\mathrm{val}(\L)<\infty.$
The last statement means that there exist $\boldsymbol{u}_{k}\in \mathbb S_+,h_{k},\ k=1,\ldots,m$
such that 
\begin{eqnarray}
\mathrm{val}(\L)= &  & \inf_{{\boldsymbol{u}_{k}\in\mathbb{S}_{+}\atop h_{k}\ge0}}\sum_{k=1}^{m}\left(w_1 u_{1k}^{\xi}+\cdots+w_d u_{dk}^{\xi}\right)^{1/\xi}h_{k}
\nonumber\\
 &  & \mathrm{subject\ to}:\ \left\{ \sum_{k=1}^{m}\max_{j\in J}\{u_{jk}\}h_{k}=c_{J}\right\} _{J\in\mathcal{J}}.\label{e:linear_program_for_I}
\end{eqnarray}
making the change of variables $a_{jk}=u_{jk}h_{k}$ gives 
\begin{eqnarray*}
\mathrm{val}(\L)= &  & \inf_{a_{jk}\ge0}\sum_{k=1}^{m}\left( w_1 a_{1k}^{\xi}+\cdots+w_d a_{dk}^{\xi}\right)^{1/\xi}\\
\text{} &  & \mathrm{subject\ to}:\ \left\{ \sum_{k=1}^{m}\max_{j\in J}\{a_{jk}\}=c_{J}\right\} _{J\in\mathcal{J}}.
\end{eqnarray*}
Thus we have proved the result for $(\L).$ The proof for
$(\U)$ follows by replacing $\sup_{A\in\mathcal{A}_{\boldsymbol{c}}} f(A)$ with $\inf_{A\in\mathcal{A}_{\boldsymbol{c}}} -f(A)$.
\end{proof}
The consequence of Theorem \ref{thm:main_redux} is that the linear semi-infinite optimization problems $(\L)$ and $(\U)$
may be reduced to finite yet \emph{non-linear} optimization problems.  Fundamentally, there is tradeoff between linearity in the semi-infinite case, 
versus non-linearity in the finite case, amounting to having to search for the \emph{finite support} of the optimal measures in $\mathrm{sol}(\L)$
and $\mathrm{sol}(\U)$.   This is because both the objective function and the constraints now depend on the unknown set of support points $T_m 
:= \{ {\bf u}_1,\cdots, {\bf u}_m\}$, in Relation \eqref{e:linear_program_for_I} in a non-linear fashion.  Note however that the size of the unknown
support $T_m$ is no greater than the number of constraints, which is one of the appealing consequences of  Theorem \ref{thm:LSIPexists} (vi).

In the case of $(\U)$, $\xi < 1$ implies that $-f(A)$ is a convex function. This, together with the fact that
$\mathcal{A}_{\boldsymbol{c}}$ is a convex set means that $\inf_{A\in\mathcal{A}_{\boldsymbol{c}}}-f(A)$
is a {\em convex optimization problem}.  Hence
$\inf_{A\in\mathcal{A}_{\boldsymbol{c}}}-f(A)$ can be solved to within arbitrary precision in polynomial 
time\cite{cvx_book}. In-practice, an exact and efficient solver for $\inf_{A\in\mathcal{A}_{\boldsymbol{c}}}-f(A)$ still needs to be developed 
and is outside the scope of this work. 

In the case of $(\L)$, $\xi < 1$ implies $\inf_{A\in\mathcal{A}_{\boldsymbol{c}}}f(A)$ is non-convex and 
generally more challenging.  However, if one makes a further assumption of {\em balanced portfolio}, i.e.
the weights in $f$ are equal $w_1= w_2 = \cdots = w_d = 1$, then further solutions are readily available
as discussed in the following section.

\subsection{Solutions for balanced portfolia}
\label{sec:closed_form}
In this section, we provide further structural results and closed form solutions in the important 
special case of {\em balanced portfolia}, where $\bfw = {\bf 1}$:
 \begin{equation}\label{e:balanced-weights}
 w_1= w_2 = \cdots = w_d = 1.
 \end{equation}
 
 \begin{rem}
 Under assumption \eqref{e:balanced-weights}, the universal dependence bounds for extreme VaR ${\cal X}_{(S,X_1)} = \rho_{\bf 1}^\xi$ given by
 \eqref{eq:universal_rhobds_le1} simplify to 
 $$
 d^\xi \le {\cal X}_{(S,X_1)} \le d.
 $$
 \end{rem}
 We show first that the minimization problem $(\L)$ reduces to a standard linear program.
 Interestingly, ${\rm val}(\L)$ is attained by spectral measures corresponding to the celebrated Tawn-Molchanov 
 max-stable models \cite{strokorb2015}.  This leads to efficient and exact solutions in practice for moderate number of constraints and dimensions.
 
 The second contribution are {\em exact formulae} for both the lower and upper bounds on $\rho:= \rho_{\bf 1}$ in the
 case when we impose only one constraint on the $d$-variate extremal coefficient
 $$
 \vartheta =  \vartheta_{\mathbf{X}} (\{1,\ldots,d\}),
 $$ 
 in addition to the standard marginal extremal coefficient constraints.   These results are possible thanks to the symmetry in the objective function when all portfolio
  weights are equal. Their proofs are given in Appendix \ref{sec:Proofs}.
 
\medskip

\begin{thm}[Tawn-Molchanov Minimizer]
\label{thm:sol_L_linprog} Under assumption \eqref{e:balanced-weights}, we have
\begin{eqnarray}
\mathrm{val}(\L) & = & 
\inf_{\boldsymbol{\beta}\in\mathbb{R}_{+}^{2^{d}-1}}\sum_{J\subset\{1,\ldots,d\},\ J\not = \varnothing}|J|^{1/\xi}\beta_{J} \label{eq:TM}\\
 &  & \mathrm{subject\ to}:\left\{ \sum_{K\subset\{1,\ldots,d\},\ K\not = \varnothing}\mathbb{I}\left\{ (K\cap J)\not=\varnothing\right\} \beta_{K}=c_{J}\right\} _{J\in\mathcal{J}}.\nonumber 
\end{eqnarray}
\end{thm}

\noindent
This result shows that obtaining the lower bound for extreme VaR in the case of a balanced portfolio amounts to solving a high-dimensional 
but standard linear program. 

\begin{rem}\label{rem:thm:sol_L_linprog} From the proof of Theorem \ref{thm:sol_L_linprog}, it follows that the lower bounds for extreme VaR in balanced portfolia 
are attained by spectral measures supported on the set of vectors 
$$
\left\{ |J|^{-1} (\boldsymbol{1}_{J}(j))_{j=1}^{d}:J\subset\{1,\ldots,d\}\right\} \subset\mathbb{S}_{+}.
$$
Such types of spectral measures correspond to the \emph{Tawn-Molchanov} max-stable model \cite{strokorb2015}.  
This is an interesting finding since, as shown in the last reference, the Tawn-Molchanov max-stable models are maximal
elements with respect to the lower orthant stochastic order, for the set of all max-stable distributions sharing a fixed set of 
extremal coefficients.  Theorem \ref{thm:sol_L_linprog}, however, is not a consequence
of the lower orthant order dominance and its proof is based on optimization results.
\end{rem}

\begin{rem}[terminology] According to a personal communication with Dr.\ Kirstin Strokorb, the Tawn-Molchanov or more completely Schlather-Tawn-Molchanov 
max-stable model originates in the works of \cite{schlather:tawn:2002} and \cite{molchanov:2008}.  In the work \cite{molchanov:strokorb:2016} it is shown to
arise from a Choquet integral and is therefore more descriptively referred to as a {\em Choquet max-stable model}.
\end{rem}

%
%
{\bf Closed form solutions.} Next, we focus on the case of a single constraint, involving the extremal coefficient 
associated with the entire set $D=\{1,\ldots,d\}.$ That is, the extremal coefficient constraints \eqref{eq:H_constraints}
in $(\L)$ and $(\U)$ are given by 
\begin{equation}\label{e:single-d-constraint}
\mathcal{J}  = \mathcal{J}_d :=  \{\{1\},\{2\},\ldots,\{d\},D\}\quad\mbox{ and }\quad
\boldsymbol{c}  = \boldsymbol{c}_\vartheta := (1,1,\ldots,1,\vartheta)\in\mathbb{R}_{+}^{d+1},
\end{equation}
where $\vartheta = \vartheta_{\mathbf{X}}(D) \in [1,d]$.
The following results show that in this special case, exploiting the
symmetry in the constraints yields {\em closed form solutions} for both ${\rm val}(\L)$ and ${\rm val}(\U)$.

\begin{thm}[lower bounds]
\label{thm:dvariate-lower} Let $B_{k}=[d(k+1)^{-1},dk^{-1})$, $k=1,\ldots , d-1.$  Under assumption \eqref{e:single-d-constraint},
for all  $\vartheta \in [1,d]$, we have that ${\rm val}(\L)$ is given by the piecewise linear function:
\begin{equation}
{\rm val}(\L) = L(\vartheta) := \sum_{k=1}^{d-1} \boldsymbol{1}_{B_{k}}(\vartheta)\left\{ \frac{k{}^{1/\xi-1}-(k+1)^{1/\xi-1}}{k^{-1}-(k+1)^{-1}}\left(\vartheta-\frac{d}{k+1}\right)+d(k+1)^{1/\xi-1}\right\}.
\label{e:d-variate-lower}
\end{equation}
 
\end{thm}

\begin{thm}[upper bounds]
\label{thm:d-variate-upper}
 Under assumption \eqref{e:single-d-constraint}, for all $\vartheta \in [1,d]$, we have
\begin{align}
{\rm val}(\U) = U(\vartheta) := \left\{ \vartheta^{\xi}+(d-1)^{1-\xi}(d-\vartheta)^{\xi}\right\} ^{1/\xi}\equiv \sup_{\boldsymbol{u}
\in\mathbb{S}_{+}}\left\{ d\left(\sum_{j=1}^{d}u_{j}^{\xi}\right)^{1/\xi}:\max_{j\in D}\{u_{j}\}=\frac{\vartheta}{d}\right\}.
\label{e:d-variate-upper}
\end{align}
 \end{thm}

%
%

%

The bounds in \eqref{e:d-variate-lower} and \eqref{e:d-variate-upper}
can be computed for arbitrary dimension and all tail index values $\xi \in (0,1]$. 
The results shown in Figure \ref{fig:d-variate}
show that the information about extreme VaR provided by a single d-variate extremal coefficient increases with the tail index $\xi$
and decreases with dimension $d$.  More concretely, computing the maximum width of the bounds 
$\sup_{\vartheta\in[1,d]} |\mathrm{val}(\U)^\xi - \mathrm{val}(\L)^\xi|$ using the 
closed form solutions and comparing to the width of the universal dependance bounds $|d - d^\xi|$ 
allow us to show that even in the high-dimensional setting of $d=100$,
with realistic tail exponent $\xi = 0.7$, the knowledge of a single $d$-variate
extremal coefficient always reduces the range of uncertainty of extreme VaR by at least $29\%$. This is a remarkable 
fact given that no other assumptions on the asymptotic dependence are imposed.

\section{Applications} \label{sec:applications}

 The goal of this section is to briefly illustrate the theoretical structural results as well as closed-form formulae 
 established above.  We start with a quantitative example of a $10$-dimensional industry portfolio, where the bi-variate
 constraints are estimated from data.  Then, in Section \ref{sec:mkt-sec}, we demonstrate how extremal coefficient constraints
 can be used to encode qualitative structural information and arrive at practical closed-form formulae.
 
\subsection{An illustration: Scale-balanced industry portfolia} \label{sec:industry-portfolia}

In this section, we briefly sketch an application of the above general results using a $d=10$-dimensional
portfolio of daily returns for $10$ industries available in \cite{french-data:2018}. The portfolio is obtained by assigning 
each of the stocks in NYSE, AMEX, and NASDAQ to one of the ten industries and then their average is computed.  Then,
a vector time-series of daily returns in percent are computed.  We shall focus on the vector time-series $\bfX_t = (X_t(j))_{j=1}^d$ 
of losses (negative returns) and study their extreme value-at-risk.   We first argue that it is reasonable to model the (multivariate) 
marginal distribution as regularly varying.  To this end, we briefly recall the standard peaks-over-threshold methodology used to 
estimate the tail index and scale of the losses.  

Let the random variable $X$ represent the loss of an asset.  The Pickands-Balkema-de Haan Theorem (see e.g.\ Theorem 3.4.13 
and page 166 in \cite{embrechts:kluppelberg:mikosch:1997}) implies that under general conditions, there exist normalizing constants $\sigma(u)>0$, such that
$$
\P\left( \frac{X-u}{\sigma(u)} > x\, {\Big\vert}\, X>u\right) \to (1+\xi x )_+^{-1/\xi},
$$
as $u\to x^*$, where $x^*:= \sup \{ x\, :\, \P(X> x) >0\}\in (-\infty,+\infty]$ is the upper end-point of the distribution of $X$.  Here $\xi \in \R$ is 
a shape parameter referred to as the tail index and $(x)_+:= \max(0,x)$.  This result suggests  that the conditional distribution 
of the distribution of the excess $X-u$ over a large threshold $u$ can be approximated with the so-called Generalized Pareto (GP) distribution, i.e.,
$$
\P( X-u > x | X >u) \approx \left(1 + \xi \frac{x}{\sigma(u)}\right )_+^{-1/\xi}.
$$
The case $\xi>0$ corresponds to heavy, power-law tails; $\xi=0$ (interpreted by continuity) is the Exponential distribution 
and $\xi<0$ is a distribution with bounded right tail.  In practice, one picks a large threshold $u$, focuses on the part of the sample exceeding
$u$, and estimates the tail index $\xi$ and scale parameter $\sigma = \sigma(u)$ via maximum likelihood applied to the excesses. (In the presence of
significant temporal dependence, extremes tend to cluster, i.e., losses occur in batches.  In this case, an important methodological step is to de-cluster the 
exceedances, i.e., to pick one observation from each cluster or otherwise reduce the dependence (see, e.g., \cite{chavez-demoulin:davison:2012}).  In our case, declustering had virtually no effect on the estimates. 

Table \ref{tab:tail-and-scale} shows the tail index and scale estimates along with their standard errors for each of the $10$ industries. They were obtained by
fitting a GP model via the method of maximum likelihood to the excesses over the $0.98$th empirical quantile, for each of the $10$ daily loss time series.
The first important observation is that all losses are heavy tailed, where the tail index estimates are not significantly different.  Indeed, the p-value of a chi-square test 
for equality of means applied to the $10$ tail index estimates (assuming normal approximation) is $0.81$.  On the other hand, the scales are significantly different 
with p-value $1.7\times 10^{-12}$.   While these marginal estimators are dependent and the chi-square test is likely to be conservative.   Therefore, with some confidence
we can assume that the daily losses have a common tail index $\xi$ and are multivariate regularly varying.  Furthermore, the GP tail asymptotics entail
 \begin{equation}\label{e:GP-tail}
 \P( X_t(j) > x) \sim p_0 \left(\frac{\sigma_j}{\xi}\right)^{1/\xi} x^{-1/\xi},\ \ \mbox{ as }x\to\infty,
 \end{equation}
 where $p_0:= 1-0.98 = 0.02$. 
\begin{table}[ht]
\centering
\caption{\label{tab:tail-and-scale} Tail index and scale estimates based on a MLE of a GPD model to 
peaks over the $0.98$-th marginal quantiles of industry losses.}
\begin{tabular}{rrrrr}
  \hline
 & $\widehat\xi$ & ${\rm s.e.}({\widehat\xi})$ & $\widehat\sigma$ & ${\rm s.e.}(\widehat \sigma)$ \\ 
  \hline
NoDur & 0.21 & 0.06 & 0.77 & 0.05 \\ 
  Durbl & 0.18 & 0.06 & 1.32 & 0.10 \\ 
  Manuf & 0.22 & 0.06 & 1.06 & 0.08 \\ 
  Enrgy & 0.19 & 0.05 & 1.05 & 0.07 \\ 
  HiTec & 0.13 & 0.05 & 1.29 & 0.09 \\ 
  Telcm & 0.22 & 0.05 & 0.84 & 0.06 \\ 
  Shops & 0.20 & 0.06 & 0.92 & 0.07 \\ 
  Hlth & 0.25 & 0.06 & 0.92 & 0.07 \\ 
  Utils & 0.14 & 0.05 & 1.21 & 0.08 \\ 
  Other & 0.14 & 0.06 & 1.25 & 0.09 \\ 
   \hline
\end{tabular}
\end{table}

 In order to apply our closed-form solutions from Section \ref{sec:closed_form}, we consider the {\em balanced} portfolio 
 $$ 
 S_t := \sum_{j=1}^d w_j X_t(j),\ \ \mbox{ with } w_j\propto \frac{1}{\widehat \sigma_j},
$$
where $\sum_{j=1}^d w_j =1$.  Thus, the scales of all assets are balanced so that $\P(w_jX_t(j) > x) \sim \P(w_1 X_t(1)>x)$ as $x\to\infty$.
Figure \ref{fig:portfolio-comparison} (left) shows the time series of daily losses for the scale-balanced portfolio.  The right panel therein shows the
empirical value-at-risk as a function of $\alpha := 1-q$ for the balanced as well as for the {\em equally weighted} portfolio $\widetilde S_t := d^{-1}\sum_{j=1}^d X_t(j)$.
\begin{figure}[h]
\centering
\includegraphics[width=6in]{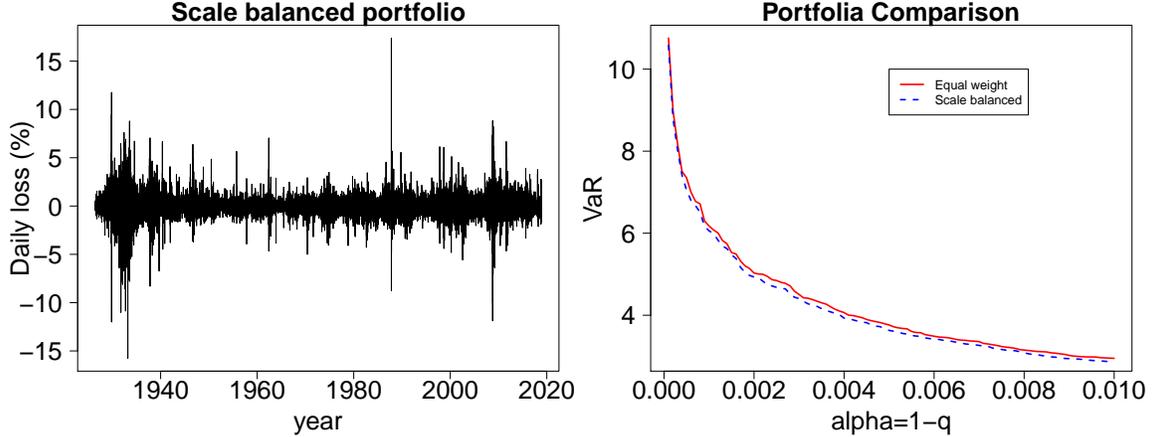}

\caption{\label{fig:portfolio-comparison} Left panel: Time series of daily losses for the scale-balanced portfolio. Right panel: empirical value-at-risk
as a function of $\alpha=1-q$ for the scale-balanced and equally weighted portfolia.}
\end{figure}

Observe that the VaR of the balanced portfolio is always lower (by about 1\% to 4.5\%) for a wide range of risk levels $q$.  This difference is significant and indicates
that the balanced portfolio is preferable in practice.  The reduction of risk may be explained by the fact that the extremal dependence in the assets is relatively balanced.  Had there been a group of industries which were significantly more dependent than the rest, the scale-balanced portfolio might not have outperformed 
the equally weighted one.  In such a case, one should balance the marginal risk (through the scales) as well as consider diversification due to extremal 
dependence.  Such portfolio optimization problems can be considered with the same tools that we employed here but they go beyond the scope of the present study.

Now, for the scale-balanced portfolio, the marginal constraints are met and one has
\begin{equation}\label{e:VaR-balanced-final}
{\rm VaR}_q (S_t) \sim \chi \times {\rm VaR}_q(w_1 X_t(1)),\ \ \mbox{ as }q\to 1,
\end{equation}
where $\chi = \rho^\xi$ with $\rho=\rho_{\bf 1}$ 
as in \eqref{e:rho_w}. Theorems \ref{thm:dvariate-lower} and \ref{thm:d-variate-upper} yield closed-form expressions for the upper and 
lower bounds on $\rho$ as a function of the single $d$-variate extremal index $\vartheta$.  On the other hand, Theorem \ref{thm:sol_L_linprog} shows that the lower bound on 
$\rho$ can be obtained by solving a linear program.  We used empirical estimates of the $d$-variate and all bi-variate extremal coefficients of the scale-balanced portfolio 
based on the $0.98$th empirical quantiles (see Table \ref{tab:theta-ests} below and Section \ref{sec:ext_coeff} for more details).  These estimates are in fact
valid extremal coefficients in the sense of Remark \ref{rem:consistent-thetas} (see Remark \ref{rem:calibration}, below).  The resulting bounds are given in 
Table \ref{tab:chi-bounds}.  Observe that the additional information in the bi-variate extremal coefficients substantially narrows the gap between the bounds based on 
a single constraint.  At the same time, relative to the wide Fr\'echet bounds, the improvement in the bounds due to single d-variate extremal coefficient is remarkable.
 \begin{table}[h]
 \centering
 \caption{\label{tab:chi-bounds} Bounds on the extreme VaR coefficient $\chi=\rho_{\bf 1}^\xi,\ \xi:= 0.1981$ for the scale-balanced portfolio with $d=10$.  }
 \begin{tabular}{c|cc}
 Constraints & Lower bound & Upper bound \\
 \hline 
Single d-variate &  4.1219 & 9.7818 \\
All bi-variate & 6.6852 & -- \\
\hline
Fr\'echet bounds (no constraints) & 1.5782 & 10
\end{tabular}
\end{table}

Finally, to obtain the estimate of ${\rm VaR}_q (S_t)$ in \eqref{e:VaR-balanced-final}, one needs to calculate the baseline  ${\rm VaR}_q(w_1 X_t(1))$.  We did so using empirical quantiles and also
from the Generalized Pareto tail approximation in \eqref{e:GP-tail}, which entails
$$
{\rm VaR}_q(w_1 X_t(1)) \approx   \frac{w_1 \widehat \xi}{\widehat \sigma_1}  \left(\frac{1-q}{p_0}\right)^{-\widehat \xi},
$$
where $\widehat \sigma_1=0.77$ and $\widehat \xi = 0.198$ is obtained through ML by assuming that the excess losses of all $10$ time series have a common 
tail index but different scales.  

\begin{figure}[h]
\includegraphics[width=6in]{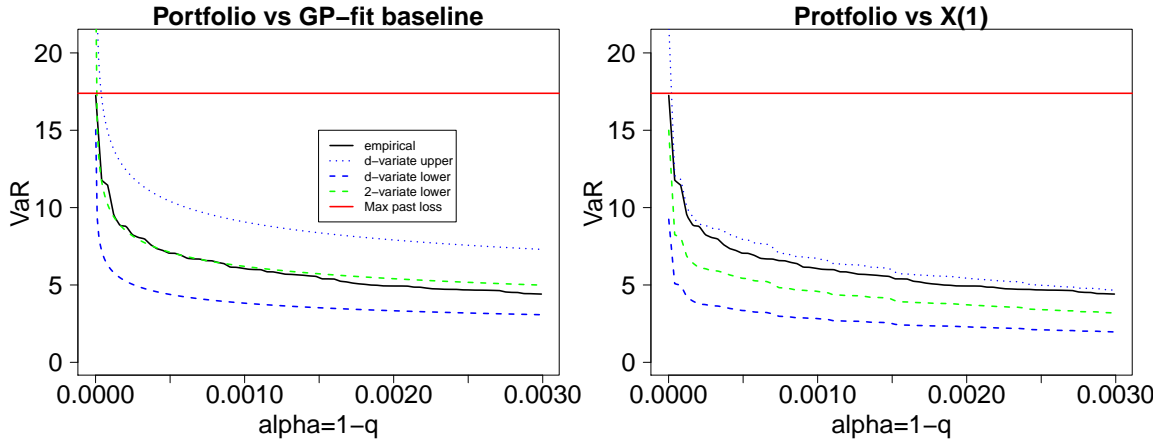}

\vspace{-.1in}
\caption{\label{fig:VaR-bounds} Upper and lower bounds on VaR as a function of $\alpha=1-q$ based on single $d$-variate and all bi-variate extremal coefficient 
constraints. The solid line indicates the empirical VaR.  Left panel: bounds are relative to the Generalized Pareto model-fit baseline.  Right panel: bounds are
relative to the empirical VaR of the non-durable goods industry.
}
\end{figure}

Figure \ref{fig:VaR-bounds} shows the upper and two types of lower bounds on ${\rm VaR}_q (S_t)$ as a function of $\alpha = 1-q$. The empirical portfolio VaR 
is also given (solid line).  The bounds in the left panel are relative to the baseline value-at-risk computed from the Generalized Pareto model approximation, 
while in the right panel ${\rm VaR}_q(w_1 X_t(1))$ is replaced by the corresponding empirical quantile.  Relative to the GP-fit baseline, the empirical portfolio VaR is within 
the upper and the larger lower bound (green dashed line) for extreme loss levels $\alpha<0.001$.  It falls slightly below the lower bound based on bi-variate 
extremal coefficient constraints for less-extreme loss levels, which can be attributed to both variability in the constraints estimates and uncertainty in the GP model.
Nevertheless, the agreement is remarkable, especially for extreme loss levels where the asymptotic approximation kicks-in.  In the right panel the bounds are relative
to the empirical value-at-risk baseline.  In this case, the portfolio VaR is {\em always} enclosed between the bi-variate lower bound and the d-variate upper 
bound and in fact the gap between them is more narrow relative to that in the left panel.  This illustrates that the asymptotic approximation 
is quite accurate for a wide range of extreme quantiles and that the extremal coefficient constraints capture well the extremal dependence between the 
assets in the portfolio.  One advantage of the GP-fit baseline however is that one can extrapolate the bounds on the portfolio VaR beyond the historically available quantile levels.
\begin{table}[h]
\centering
\caption{\label{tab:return-levels} Bounds on the return levels for the scale-balanced portfolio. VaR$_q$ with $1- q =  1/ (252\times m)$ is exceeded 
on the average once in every $m$ years. }
\begin{tabular}{r|rrr}
   Return levels (years) & 10 & 100 & 1000 \\ \hline
  d-variate upper   & 10.90 & 17.20 & 27.14 \\ 
  d-variate lower & 4.59 & 7.25 & 11.44 \\ 
  bi-variate lower & 7.45 & 11.75 & 18.55 \\ 
   \hline
\end{tabular}
\end{table}
 Indeed, Table \ref{tab:return-levels} provides bounds on the $10$, $100$ and $1000$-year return levels, where a year is assumed to have $252$ trading days.  
These results indicate for example that one should expect to encounter daily losses exceeding $4.59\%$ once in $10$ years on the average, even for the relatively 
diversified scale-balanced portfolio, but daily losses of $17.2\%$ or more are unusual 1-in-a-100 year type events.  Even though these results hinge on the 
assumption of stationarity in the  extremal dependence structure, they provide novel distributionally robust bounds of extreme portfolio or insurance risk 
and can be used to validate most if not all other model-based estimators of extreme value-at-risk.

\begin{table}[ht]
\centering
\caption{\label{tab:theta-ests} Empirical estimates for the bivariate extremal coefficients $\widehat \vartheta(\{i,j\})$
for the scale-balanced 10-industry potfolio based on exceedances over the $0.98$th quantiles.
See \eqref{e:theta-hat}.  The single $d$-variate extremal coefficient 
estimate based on the same quantile is $\widehat \vartheta(\{1,\dots,d\}) = 3.15$.}
\begin{tabular}{rrrrrrrrrrr}
  \hline
  {\ } & NoDur & Durbl & Manuf & Enrgy & HiTec & Telcm & Shops & Hlth & Utils & Other \\ 
  \hline
  NoDur & 1.00 & 1.46 & 1.37 & 1.50 & 1.48 & 1.54 & 1.38 & 1.44 & 1.48 & 1.40 \\ 
  Durbl & 1.46 & 1.00 & 1.30 & 1.50 & 1.46 & 1.58 & 1.44 & 1.57 & 1.46 & 1.35 \\ 
  Manuf & 1.36 & 1.29 & 1.00 & 1.43 & 1.40 & 1.53 & 1.36 & 1.49 & 1.40 & 1.26 \\ 
  Enrgy & 1.49 & 1.50 & 1.43 & 1.00 & 1.60 & 1.61 & 1.52 & 1.60 & 1.54 & 1.47 \\ 
  HiTec & 1.48 & 1.45 & 1.40 & 1.60 & 1.00 & 1.55 & 1.43 & 1.55 & 1.47 & 1.45 \\ 
  Telcm & 1.53 & 1.58 & 1.54 & 1.61 & 1.55 & 1.00 & 1.55 & 1.60 & 1.61 & 1.52 \\ 
  Shops & 1.37 & 1.44 & 1.36 & 1.52 & 1.43 & 1.55 & 1.00 & 1.48 & 1.47 & 1.38 \\ 
  Hlth & 1.43 & 1.56 & 1.49 & 1.60 & 1.55 & 1.60 & 1.48 & 1.00 & 1.60 & 1.54 \\ 
  Utils & 1.47 & 1.46 & 1.40 & 1.55 & 1.47 & 1.61 & 1.47 & 1.60 & 1.00 & 1.44 \\ 
  Other & 1.39 & 1.35 & 1.26 & 1.47 & 1.45 & 1.52 & 1.38 & 1.54 & 1.44 & 1.00 \\ 
   \hline
\end{tabular}
\end{table}


\subsection{Market and sectors framework} \label{sec:mkt-sec}

While the quantitative methods in previous section based on the knowledge of all bi-variate constraints yield tight bounds, their use in practice is 
limited to small and moderate dimensions due to practical challenges in solving the optimization problems.  In this section, our goals is two-fold. First, 
we illustrate how one may encode structural/expert knowledge through extremal 
 coefficient constraints.  Secondly, we show that the closed-form expressions in Theorems \ref{thm:dvariate-lower} and \ref{thm:d-variate-upper}
 can lead to practical and tight bounds on extreme VaR in high-dimensions, where numerical optimization is either challenging or impossible.
 
  We do so over a simple but instructive `market plus sectors' framework.  Namely, 
 suppose that the vector of portfolio losses $\bfX = (X_1,\cdots,X_d)$ is regularly varying with index $\xi \in (0,1)$ 
 and standardized marginal scales in the sense of \eqref{e:Ki-standardization}.   
 
 Let $\beta\in (0,1)$, and suppose that
 \begin{equation}\label{e:mkt-and-sectors}
  \bfX = \beta^\xi \bfX_{\rm mkt} + (1-\beta)^\xi \bfX_{\rm sec}, 
 \end{equation}
 where $\bfX_{\rm mkt}$ and $\bfX_{\rm sec}$ are independent and also regularly varying with index $\xi$ and 
asymptotically standardized margins (as in \eqref{e:Ki-standardization}). The components $\bfX_{\rm mkt}$ and $\bfX_{\rm sec}$ 
represent  the overall market and individual sector-specific risks, respectively.

We shall assume that the market risk affects all stocks and therefore model it as asymptotically {\em completely dependent}, i.e.,
$$
{\vartheta}_{\bfX_{\rm mkt}} (\{1,\dots,d\}) = 1.
$$
We shall also assume that $\bfX_{\rm sec} = (\bfX(1),\dots,\bfX(k))$ is partitioned into independent sub-vectors
$\bfX(i) = (X_{j} (i))_{j=1}^{d_i}$, each corresponding to a sector.  That is, $d= d_1+\cdots+d_k$ and
$$
\{1,\dots,d\} = J_1\cup\cdots \cup J_d,
$$
where $J_i,\ i=1,\dots,k$ are pairwise disjoint sets of indices.

Relation \eqref{e:mkt-and-sectors} leads to a simple but natural 2-tier asymptotic dependence structure.  The parameter $\beta$ controls the proportion of risk due to
 overall market-wide events, while the individual sectors may experience independent, and largely arbitrary internal risk exposures accounted for
 by the sector-specific component.  We demonstrate next how the closed form formulae in Theorems \ref{thm:dvariate-lower} and \ref{thm:d-variate-upper} lead to 
 tight  lower- and upper-bounds on extreme VaR for the portfolio $\bfX$.

Using the independence of the market and the sectors, it can be shown that:
\begin{equation}\label{e:mkt-sec-H}
H_{\bfX} = \beta\times H_{\bfX_{\rm mkt}} + (1-\beta) \times \sum_{i=1}^k H_{\bfX(i)},
\end{equation}
where $H$ with the corresponding sub-script is the properly normalized spectral measure of the corresponding 
vector $\bfZ:= \bfX^{1/\xi}$, where we naturally embed $H_{\bfX(i)}$ into the higher-dimensional space 
$\mathbb S_+ \subset \mathbb R^d$.

In view of \eqref{e:rho_w}, Relation \eqref{e:mkt-sec-H} entails
$$
  \rho_{\bfX} = \beta \times \rho_{\rm mkt} + (1-\beta) \times \sum_{i=1}^k \rho_{{\rm sec}, i},
$$
where $\rho_{\bfX} = \rho_{\bf 1}(H_{\bfX},\xi)$, $\rho_{\rm mkt} = \rho_{\bf 1}(H_{\bfX_{\rm mkt}},\xi)$, and
$\rho_{{\rm sec}, i} = \rho_{\bf 1}(H_{\bfX(i)},\xi)$ are the corresponding $\rho$-functionals of the overall
portfolio, its market, and sector components, respectively.  

Similarly, in view of \eqref{e:ext-coeffs}, Relation \eqref{e:mkt-sec-H} implies that
for every $J\subset\{1,\dots,d\}$, we have
\begin{equation}\label{e:mkt-sec-theta-J}
\vartheta_{\bfX} (J) = \beta\times \vartheta_{\bfX_{\rm mkt}}(J) + (1-\beta)\times \sum_{i=1}^k \vartheta_{\bfX(i)}(J).
\end{equation}

Notice that $\vartheta_{\bfX_{\rm mkt}}(J) = 1$, for all non-empty sets $J$, since the market factor is completely dependent.

These decomposition results allow us to obtain closed-form lower- and upper-bounds on $\rho_{\bfX}$ in terms of $\beta$ and $\vartheta_{\bfX}(J_i),\ i=1,\dots,d$.
Indeed, we have:
\begin{equation}\label{e:theta_sector_i}
\vartheta_{\bfX(i)}(J_i) = (\vartheta_{\bfX}(J_i) -\beta)/(1-\beta),\ i=1,\dots,k.
\end{equation}
Now, using the closed-form expressions for $\rho_{\bfX(i)}$ for each of the sectors $i=1,\dots,k$ based on the single $d_i$-variate constraint 
$\vartheta_{\bfX(i)}(J_i)$, for $i=1,\dots,k$, we obtain
\begin{equation}\label{e:mkt-sect-bounds}
{\cal B}(\bfX) = \beta\times d^{1/\xi} + (1-\beta) \times \sum_{i=1}^k {\cal B}\Big(d_i, \xi, (\vartheta_{\bfX}(J_i) -\beta)/(1-\beta) \Big),
\end{equation} 
where ${\cal B} \in \{ {\cal L}, {\cal U}\}$ and $ {\cal B}(d, \xi,\vartheta)$
denotes either the lower- or upper-bound formulae from \eqref{e:d-variate-lower} or \eqref{e:d-variate-upper}.

\begin{figure}[H]
\centering
\includegraphics[width=6in]{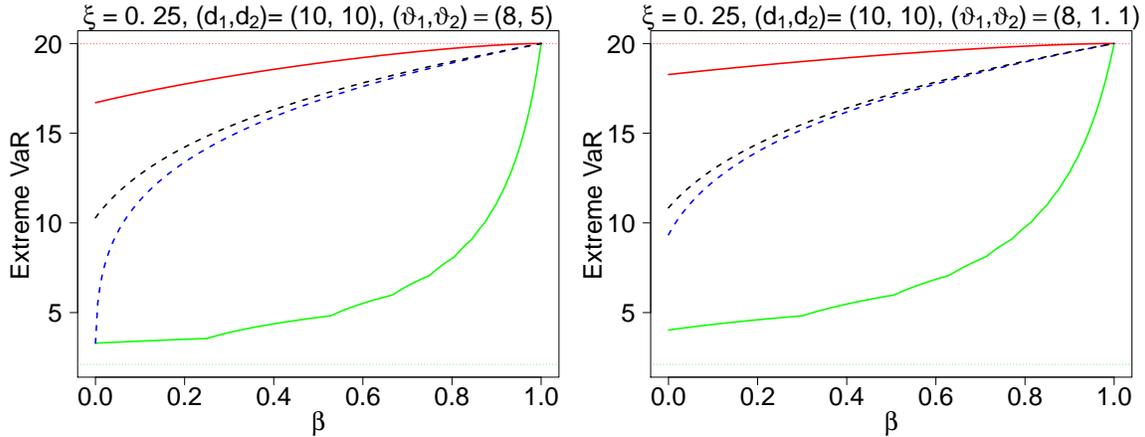}
\protect\caption{Upper and lower bounds on extreme VaR for a composite market plus sectors portfolio.
The dashed lines indicate the closed-form expressions based on \eqref{e:mkt-sect-bounds}. The solid lines
indicate the conservative bounds based on a single $d$-variate extremal coefficient constraint without any 
structural assumptions.}\label{fig:mkt-sectors}
\end{figure}
 
Figure \ref{fig:mkt-sectors} illustrates the significant reduction in the range of possible extreme VaR values based on the above setup for a range of $\beta$-values.
We have the simple partition into $k=2$ sectors and $(d_1,d_2) = (10,10)$.  Considered are two cases where the within-sector $d_i$-variate extremal coefficients 
in \eqref{e:theta_sector_i} are $(\vartheta_{\bfX(1)}(J_1),\vartheta_{\bfX(2)}(J_2)) = (8,5)$  (left panel) and $(\vartheta_{\bfX(1)}(J_1),\vartheta_{\bfX(2)}(J_2)) = (8,1.1)$
 (right panel).  In both cases $\xi =0.25$.
The dotted horizontal lines indicate the Hoeffding-Fr\'echet bounds on extreme VaR. The solid green and red lines correspond to the exact lower/upper bounds obtained by imposing a single $d$-variate constraint with
$$
\vartheta_{\bfX}(\{1,\dots,d\}) = \beta + (1-\beta) \times \sum_{i=1}^k \vartheta_{\bfX(i)}(J_i).
$$
Finally, the dashed blue/black lines correspond to the lower/upper bounds for $\rho$ obtained by using the decomposition into a single market factor effect
plus independent sector-specific risks.  Observe the significant reduction in the range of possible values for extreme VaR.  This is naturally attributed to the 
assumption of independence among the sectors.  The presence of a single asymptotically completely dependent market factor, however, can make this 
range approach the ultimate upper bound for $\beta \to 1$. Alternatively, if the proportion of the market risk is low ($\beta\to 0$), the lower bound approaches 
the ultimate single d-variate constraint lower bound (solid green curve) in the left panel.  In the right panel, however, one observes a non-trivial gap between the two lower 
bounds at $\beta=0$.  This can be attributed to the fact that the constraint $\vartheta_{\bfX(2)}(J_2)=1.1$ is rather close to complete dependence for the second sector, while
the overall portfolio constraint on $\vartheta_{\bfX}(\{1,\dots,d\})$ is far from complete dependence.  Thus, the additional sector-specific information leads to far less optimistic 
lower bound on extreme VaR than in the market-structure-agnostic case. 

The proportion of market-wide risk $\beta$ here was assumed to be known, for illustration purposes.  Using \eqref{e:mkt-sec-theta-J}, 
however, $\beta$ can be readily estimated in practice from an extremal coefficient $\vartheta_\bfX(J)$ involving a set $J$ 
of two or more sectors.  For example, given $\vartheta_{\bfX}(\{1,\dots,d\}) = c_0(\bfX)$ and $\vartheta_{\bfX}(J_i) = c_i(\bfX),\ i=1,\dots,k$, 
we obtain
$$
c_0(\bfX) = \beta + (1-\beta)\times \sum_{i=1}^k \vartheta_{\bfX(i)}(J_i)\quad \mbox{ and }\ c_i(\bfX) = \beta + (1-\beta) \times \vartheta_{\bfX(i)}(J_i),\ i=1,\dots,k.
$$
By elimination, these linear equations yield
$$
 \beta = \frac{\Big( \sum_{i=1}^k (c_i(\bfX) \Big) - c_0(\bfX)}{k-1},\ \ \mbox{ as well as }\ \ \vartheta_{\bfX(i)}(J_i) = \frac{c_i(\bfX) - \beta}{1-\beta}.
$$

\section{Summary and discussion}

 Under the general assumption of multivariate regular variation, the {\em extreme Value-at-Risk} of a $d$-dimensional portfolio, relative to a baseline asset, 
 can be expressed as an integral functional with respect to a finite measure on the unit simplex.  This, unknown (spectral) measure, is an infinite-dimensional
 parameter that encodes the complete extremal (joint) dependence structure of the assets in the portfolio.  In practice, the conventional 
 estimation of the spectral measure is challenging or impossible.  This motivated us to adopt distributionally robust perspective.  Namely,  study the
 optimization problems of finding the infimum and supremum of the extreme VaR functional over large classes of possible spectral measures.  Using
 popular and interpretable extremal coefficient constraints, we expressed the above optimization problems as duals to linear semi-infinite programs, which in turn
 were shown to have no duality gap.  Thus, a number of results on the structure of spectral measures corresponding to the best- and worst-case extreme VaR were
 obtained.  In the special case of {\em scale balanced portfolia}, we have also shown that the lower bound on extreme VaR corresponds to a spectral 
 measure of the so-called Tawn-Molchanov multivariate max-stable model, which can be solved with conventional linear programs.  We have also established surprising
 closed-form expressions for the lower- and upper-bound on extreme VaR under single $d$-variate extremal coefficient constraints, valid in all dimensions $d$.  These results
 were further illustrated and extended in the case of the market-and-sectors framework. The theoretical results were shown to provide practical bounds in a limited real data example, 
 and compared with conventional extreme value theory method.
 
 Our contributions are mostly theoretical. However, the established results, formulae and methods are motivated by important challenges in quantifying model uncertainty 
 when studying the risk of extremes in high-dimensional portfolia.  To provide a complete practical methodology for risk assessment a number of implortant
 problems remain to be addressed.  Namely,
 
  \begin{itemize}
    \item Develop practical or approximate solvers for the optimization problems in dimensions $d>10$.
    \item Study the optimal set of constraints ${\cal J}$ in terms of greatest reduction of the range of possible extreme VaR.
    \item Quantify the uncertainty in the resulting lower- and upper-bounds on extreme VaR stemming from the statistical error in 
    the estimation of the tail index $\xi$ and extremal coefficient constraints. 
  \end{itemize}  
  
  Finally, one very important open problem that stands out in our opinion is to establish closed form formulae in the case of single $d$-variate constraints
  (as in Theorems \ref{thm:dvariate-lower} and \ref{thm:d-variate-upper}) for a general un-balanced portfolio.  Such formulae, by the method of partitioning, 
  can lead to significant improvements on the range of extreme VaR similar to the ones obtained in Section \ref{sec:mkt-sec}.

 \section*{Acknowledgements} We thank two anonymous referees and an Associate Editor for their constructive criticism, which helped us significantly
 improve the presentation. Section \ref{sec:mkt-sec} was motivated by a question raised by a referee.  We also thank Dr.\ Kirstin Strokorb for mathematical 
 insights on Tawn-Molchanov max-stable models and for inspiring discussions. 
  
\appendix

\section{\label{sec:MRV_EVT} Multivariate regular variation and extremes}

For convenience of the reader, here we review some facts and technical results
on multivariate regular variation and extremes. For more details, see the comprehensive monographs 
\cite{resnick:1987,dehaan:ferreira:2006,resnick:2007} and the recent general approach to regular 
variation in metric spaces \cite{hult:lindskog:2006}. Some applications and extensions can be found 
in \cite{lindskog:resnick:roy:2014} and \cite{scheffler:stoev:2014}.

\begin{defn}\label{d:MRV} A random vector $\mathbf X = (X_i)_{i=1}^d$ in $\R^d$ is said to be multivariate regularly varying (MRV), if
there exist a sequence $a_n\ge 0,\ a_n\uparrow\infty$ and a Borel measure $\mu$ on $\R^d\setminus\{\mathbf 0\}$, such that:

{\it (i)} $\mu(A)<\infty$, for all Borel sets $A$, bounded away from the origin, i.e., such that $A\subset \R^d\setminus B(\mathbf 0,\epsilon),$ for some $\epsilon>0$,
where $B(\mathbf 0,\epsilon)$ denotes a ball centered at $\mathbf 0$ with radius $\epsilon$.

{\it (ii)} For all Borel sets $A$, bounded away from $\mathbf 0$ and such that $\mu(\partial A) = 0$, we have
\begin{equation}\label{e:d:MRV}
 n\P(a_n^{-1}\mathbf X \in A) \longrightarrow \mu(A),\ \ \mbox{ as }n\to\infty.
\end{equation}
In this case, we write $\mathbf X \in RV(\{a_n\},\mu)$.
\end{defn}
It can be shown that if $\mathbf X \in RV(\{a_n\},\mu)$, the sequence $a_n$ is necessarily regularly varying, i.e.\ there exist a positive constant 
$\xi>0$, such that  $a_{[tn]}/a_n \to t^{\xi},$ as $n\to\infty$, for all $t>0$.  Furthermore, the limit measure $\mu$ has the scaling 
property $\mu(c A) = c^{-1/\xi} \mu(A)$, for all $c>0$.  Different choices for the normalization sequence $\{a_n\}$ are possible, however,
the exponent $\xi$ is uniquely defined, given a random vector $\bfX$. To indicate that, we sometimes write $\bfX \in RV_{1/\xi}(\{a_n\},\mu)$.

An alternative, equivalent approach to multivariate regular variation is through polar coordinates.  
Namely, let $\|\cdot\|$ be an arbitrary norm in $\R^d$ (In fact, one can consider any positive and $1$-homogeneous continuous 
function on $\R^d\setminus\{\mathbf 0\}$ as the {\em radial} component see, e.g.,\ \cite{scheffler:stoev:2014}.) 
Then, $\mathbf X\in RV(\{a_n\},\mu)$, if and only if, for any (all) $s>0$,
\begin{equation}\label{e:MRV-polar}
 n\P( a_n^{-1} \|\mathbf X\| >s,\ \mathbf X/\|\mathbf X\| \in \cdot) \stackrel{w}{\longrightarrow} c s^{-1/\xi} \sigma(\cdot),\ \ \mbox{ as }n\to\infty,
\end{equation}
for some {\em probability measure} $\sigma$ defined on the unit sphere ${\mathbb S}_{\|\cdot\|}:=\{ \bfx \, :\, \|\bfx\|=1\}$.  It can be easily seen from 
\eqref{e:d:MRV} and \eqref{e:MRV-polar}, by setting $s=1$, that
$$
 c = \mu(\{\|\mathbf x\| >1\}) \ \ \mbox{ and, in fact, }\ \ \sigma(B)= \frac{\mu(\{\|\bfx\|>1, \bfx/\|\bfx\| \in B\})}{\mu(\{\|\bfx\|>1\})}.  
$$
for a Borel set $B\subset{\mathbb S}_{\|\cdot\|}.$
Relation \eqref{e:MRV-polar} can be interpreted in terms of polar coordinates as follows. Letting $\bfX\rightsquigarrow (R,\mathbf U)$ with 
$R:= \|\mathbf X\|$ and $\mathbf U:= \mathbf X/\|\mathbf X\|$, we have that 
$$
 n\P(a_n^{-1} R >s ) \to c s^{-1/\xi} \ \ \mbox{ and } \ \ \P( \mathbf U \in \cdot  | R>a_n ) \stackrel{w}{\longrightarrow} \sigma(\cdot),
$$
as $n\to\infty$.  This means, that the vector $\mathbf X  \rightsquigarrow(R,\mathbf U)$ is MRV if and only if its radial component is regularly varying and the conditional distribution 
of its angular component, given that the radius is extreme, converges weakly to the probability measure $\sigma$ (see, e.g., \cite{hult:lindskog:2006} and
Prop 3.9 in \cite{scheffler:stoev:2014}).  The {\em probability} measure $\sigma$ is referred to as the {\em
spectral measure} of $\mathbf X$.  Observe that, depending on the choice of the normalizing sequence $\{a_n\}$, the measure $\mu$ in \eqref{e:d:MRV}
and correspondingly, the constant $c$ in \eqref{e:MRV-polar}, may change.  The spectral measure $\sigma$ and the exponent $1/\xi$, however, are uniquely defined,
given a RV vector $\mathbf X$.

The measure $\mu$ has the polar coordinate representation $\mu(d\bfx) = c \nu_{1/\xi}(dr) \sigma(d\bfu)$, where is $\nu_{1/\xi}$ is a measure on $(0,\infty)$, such that
$\nu_{1/\xi}(c,\infty) = c^{-1/\xi},\ c>0$.  More precisely, we have the {\em disintegration formula}:
\begin{equation}\label{e:disintegration}
 \mu(A) = c \int_{\mathbb S_{\|\cdot\|}} \int_0^\infty  1_A(r\bfu) (1/\xi)  r^{-1-1/\xi} dr \sigma(d\bfu).
\end{equation}

\subsection{\label{sec:mevt}Multivariate extremes}

In the context of extreme value theory, the spectral measure $\sigma$ can be used to express the cumulative distribution function of the 
asymptotic distribution of independent component-wise maxima.  Specifically, let $\mathbf X = (X_i)_{i=1}^d,\ \mathbf X(k),\ k=1,\ldots,n$ be 
iid RV$(\{a_n\},\mu)$. For simplicity, assume that the $X_i$'s are non-negative. Then, the measure $\mu$ concentrates on $[0,\infty)^d\setminus\{\mathbf 0\}$. 
Consider the component-wise maxima $M_i(n) := \max_{k=1,\ldots,n} X_i(k),\ i=1,\ldots,d$.  Then, it can be shown that 
for all $\bfx =(x_i)_{i=1}^d\in [0,\infty)^d\setminus\{\mathbf 0\}$,
\begin{equation}\label{e:Frechet-limit}
 \P{\Big(} a_n^{-1} M_i(n) \le x_i,\ i=1,\ldots,d {\Big)} \longrightarrow G_\mu(\bfx):= \exp{\Big\{}-\mu([\mathbf 0,\bfx]^c){\Big\}},\ \ \mbox{ as }n\to\infty.
\end{equation}
That is, $a_n^{-1} \mathbf M_n := a_n^{-1} (M_i(n))_{i=1}^d$ converges in distribution to a vector $\mathbf Y$ with the cumulative distribution function $G_\mu$
given above. Indeed, by the independence of the $\mathbf X(k)$'s, we have
\begin{equation}\label{e:Frechet-limit-1}
 \P{\Big(} a_n^{-1} \mathbf M(n) \le \bfx {\Big)} = \P  {\Big(} a_n^{-1} \mathbf X \le \bfx {\Big)}^n = {\Big(} 1 - \frac{n\P(a_n^{-1} \mathbf X \in A)}{n}{\Big)}^n, 
\end{equation}
where $A = [\mathbf 0,\mathbf x]^c$, $\mathbf M(n) = (M_i(n))_{i=1}^d$ and the above inequalities are considered component-wise.
By using the scaling property of $\mu$, it can be shown that $A$ is a continuity set, and hence \eqref{e:d:MRV}
implies that $n\P(a_n^{-1} \mathbf X \in A) \to \mu(A)$, as $n\to\infty$.  Hence, the right-hand side of \eqref{e:Frechet-limit-1} converges to
$\exp\{-\mu(A)\}$, which is in fact the right-hand side of \eqref{e:Frechet-limit}.

Consider now the disintegration formula \eqref{e:disintegration} with $A = [\mathbf 0,\mathbf x]^c$.  Notice that
$r\bfu \in A$ if $r u_i > x_i$, for some $i=1,\ldots,d$, or equivalently $r>\min_{i=1,\ldots,d} x_i/u_i$.  Therefore, by \eqref{e:disintegration}, we have
$$
\mu(A) = c\int_{\mathbb S_{\|\cdot\|}} \int_{\min_{i=1,\ldots,d} x_i/u_i} (1/\xi) r^{-1-1/\xi} dr \sigma(d\bfu) = c \int_{\mathbb S_{\|\cdot\|}} {\Big(}
  \max_{i=1,\ldots,d} {u_i\over x_i} {\Big)}^{1/\xi} \sigma(d\bfu).
$$
That is, we obtain the following well-known expression of the distribution function $G_\mu$:
\begin{equation}\label{e:Frechet-spectral}
\P(\mathbf Y\le \bfx) \equiv G_\mu(\bfx) = \exp {\Big\{}-c \int_{\mathbb S_{\|\cdot\|}} {\Big(}\max_{i=1,\ldots,d} {u_i \over x_i} {\Big)}^{1/\xi} \sigma(d\bfu){\Big\}},
\end{equation}
$\bfx \in \mathbb R_+^d\setminus \{\mathbf 0\}$ (see, e.g., Ch.\ 5 in \cite{resnick:1987}).

\subsection{\label{sec:ext_coeff} Extremal coefficients} Let $J \subset\{1,\ldots,d\}$ be a non-empty subset of coordinates of 
the random vector $\mathbf Y$ in \eqref{e:Frechet-spectral}.  Recall that the extremal coefficient ${\vartheta}(J)$ is defined as follows
$$
\P {\Big(}\max_{j\in J} Y_j\le 1{\Big)} =: \exp\{ -{\vartheta}(J) \}.
$$
In view of \eqref{e:Frechet-spectral}, we have
\begin{equation}\label{e:theta-def}
{\vartheta}(J) = c \int_{\mathbb S_{\|\cdot\|}} {\Big(}\max_{j\in J} u_i^{1/\xi} {\Big)} \sigma(d\bfu).
\end{equation}
Moreover, by \eqref{e:Frechet-limit} one can show that 
$$
n \P {\Big(}\max_{j\in J} X_j >a_n x{\Big)} \longrightarrow {\vartheta}(J),\ \mbox{ as }n\to\infty.
$$
Therefore, modulo a common scaling factor, all these extremal coefficients can be readily estimated via the asymptotic scale 
coefficients of the heavy-tailed distributions $\max_{j\in J} X_j$.  Specifically, we have
$$
 \lim_{x\to\infty}\frac{\P (\max_{j\in J} X_j >x)}{\P(X_1>x)}  = {{\vartheta} (J) \over {\vartheta}(\{1\})},\ \ J\subset \{1,\ldots,d\}.
$$
By suitable rescaling of the {\em reference asset} $X_1$ (or equivalently, the normalization sequence $\{a_n\}$), without loss of generality,
we may assume that $\vartheta(\{1\}) = 1.$ Given independent copies $\mathbf X_i,\ i=1,\ldots,n$ of $\mathbf X$, define the self-normalized 
estimators
\begin{equation}\label{e:theta-hat}
  \widehat\vartheta_{x}(J):= {\sum_{i=1}^n \mathbb{I}( \max_{j\in J} X_j(i) >x) \over \sum_{i=1}^n \mathbb{I}( X_1(i) >x)}.
\end{equation}
\begin{rem}
It can be shown that the estimators in \eqref{e:theta-hat} are weakly consistent for any choice of a regularly varying sequence $x=x_n\to \infty$ such that 
$n\P(X_1(i)> x_n)\to \infty$, as $n\to\infty$, i.e., we have $\hat\vartheta_{x_n}(J)\to \vartheta(J)$ in probability. This is true for example for the sequence 
$x_n:= n^{1/(1/\xi+\delta)},$ for any $\delta>0$. The consistency of $\hat\vartheta_{x_n}(J)$ follows by applying 
Theorem 5.3.(ii) in \cite{resnick:2007} to both the numerator and denominator in \eqref{e:theta-hat}, viewed as 
empirical measures of the type $b_n^{-1} \sum_{i=1}^n {\mathbb I}_{\{ Y(i)/x_n\}} (\cdot)$, where $Y(i)$ stands for either $X_1(i)$ or $\max_{j\in J}X_j(i)$.  
The sequence $b_n\nearrow \infty$ herein is chosen such that $(n/b_n)\P( X_1(i) > x_n s) \to s^{-1/\xi},$ as $n\to\infty$, for all $s>0$.  The fact 
that such a sequence $b_n$ can be found follows from the regular variation property of $x_n$ and the distribution of $X_1(i)$.
\end{rem}

\begin{rem} \label{rem:calibration} Recall Remark \ref{rem:consistent-thetas}.  The empirically estimated extremal coefficients in \eqref{e:theta-hat} do satisfy the
 consistency relationships of a set of valid extremal coefficients.  Indeed, it follows from Lemma \ref{lem:max-alteration} (below), with $x_j:= \mathbb{I}(X_j(i)>x)$ that
 $$
  \sum_{L \,:\, J\subseteq L \subseteq \{1,\dots,d\}} (-1)^{|L\setminus J|+1}  \max_{j \in L} \mathbb{I}\Big(X_j(i) > x\Big)  \equiv
  \sum_{L \,:\, J\subseteq L \subseteq \{1,\dots,d\}} (-1)^{|L\setminus J|+1} \mathbb{I}\Big( \max_{j \in L} X_j(i) > x \Big) \ge 0,
 $$
 for all $i$.  Thus, the desired consistency relationships in Remark \ref{rem:consistent-thetas}, follow by summing over $i$ since the denominator 
 in \eqref{e:theta-hat} is common and positive. 
\end{rem}

\begin{rem} In practice, however, when the extremal coefficients are either imposed or estimated in some other way, different from \eqref{e:theta-hat}, 
 one needs to ensure they provide consistent constraints.  This can be done by ``projecting" them onto the convex set of valid vectors of extremal 
 coefficients ${\bf c}_{\cal J} = (c_J)_{J\in {\cal J}}$.
   Specifically, by M\"obius inversion, we know that ${\bf c}_{\cal J} = A \boldsymbol{\beta}$, where $\boldsymbol{\beta} \in \mathbb R_+^{2^d-1}$ and
   a certain design matrix $A$ of dimension $|{\cal J}| \times (2^{d}-1)$.  In practice, if the vector of estimated coefficients is $\widehat {\bf c}_{\cal J}$,
   we solve the quadratic optimization program 
   $$
   {\rm minimize}\  \Big\{   \| \widehat {\bf c}_{\cal J} - A{\boldsymbol{\beta}}\|^2 + \lambda \|\boldsymbol{\beta}\|^2 \Big\},
   $$
   subject to $\boldsymbol{\beta}\ge {\bf 0}$, for some small regularization parameter $\lambda>0$.  We take the solution $A\boldsymbol{\beta}$ as the constraints in our 
   extreme VaR optimization algorithms. In our experience, the so-calibrated extremal coefficient constraints are quite close to the ones estimated 
   in practice.  This calibration and other important statistical issues merit further independent investigation.
\end{rem}

The following elementary lemma follows by induction, although it may be possible to obtain with general M\"obius inversion techniques.  This result is
used to show that the empirical extremal coefficients in \eqref{e:theta-hat} satisfy the consistency relationships of a valid set of extremal coefficients 
(cf Remark \ref{rem:consistent-thetas}). 

\begin{lem}\label{lem:max-alteration} Let $d\ge 2$ be an integer.  For all $x_i\ge 0,\ i=1,\dots,d$, and 
$J \subset\{1,\dots,d\},\ J \not=\{1,\dots,d\}$,
we have
$$
S(J):= \sum_{L \,:\, J\subseteq L \subseteq \{1,\dots,d\}} (-1)^{|L\setminus J|+1} x_j \ge 0,
$$
where by convention $\max_{j\in \emptyset}x_j :=0$. 
\end{lem}
\begin{proof} We establish the claim by induction.  If $J=\{1,\dots,d\}\setminus\{j_0\}$, then trivially $S(J) = 0$, if $\max_{j\in J} x_j \ge   x_{j_0}$ and 
$S(J) = x_{j_0} - \max_{j\in J} x_j >0$, otherwise.  This proves that $S(J)\ge 0$, for all $J$ such that $|J|=d-1$. 

Suppose, now that $|J| \le d-2$ and $S(\widetilde J)\ge 0$, for all $|\widetilde J| \ge |J|+1$.  
Let $\{j_1,\dots,j_m\}:=\{1,\dots,d\}\setminus J $ and observe that
$$
S(J) = - \max_{j\in J} x_j + \sum_{i=1}^m \max_{j\in J\cup\{j_i\}} x_j + \sum_{i=1}^m S(J\cup\{j_i\}).
$$
The latter however is non-negative.  Indeed, by the induction assumption, we have $S(J\cup\{j_i\})\ge 0$, while $\max_{j\in J\cup\{j_i\}}\ge \max_{j\in J} x_j$, for each
$i=1,\dots,m$, which since the $x_j$'s are non-negative implies that $S(J)\ge 0$.  Appealing to the induction principle, we conclude the proof.
\end{proof}

\subsection{On Extreme VaR for homogeneous risk functionals}

Let $\mathbf X\in RV_{1/\xi}(\{a_n\},\mu)$ be a vector of losses.  It is convenient to write $\bfX = (Z_i^{\xi})_{i=1}^d$, where 
$\bfZ = (Z_i)_{i=1}^d \in RV_{1}(\{b_n\},\nu)$, with $b_n:= a_n^{1/\xi}$ and $\nu (A) = \mu (A^{\xi})$.  

Consider a set of positive portfolio weights $w_i>0,\ i=1,\ldots,d$ for the $d$ assets. Then, the cumulative portfolio loss
$S=\sum_{i=1}^d w_i X_i$ can be expressed as 
$$
  S = h_{\bfw} (\bfZ),\ \ \mbox{ where }\ h(\bfz) = \sum_{i=1}^d w_i z_i^{\xi} \equiv \sum_{i=1}^d w_i z_i^\xi,
$$
is a positive, $\xi$-homogeneous function of $\bfZ$. 

The asymptotic scale of the loss $S$ relative to a reference asset is the key ingredient in computing 
extreme Value-at-Risk. Indeed, if
\begin{equation}\label{e:rho-S-X1}
\rho := \lim_{x\to\infty} \frac{\P(S>x)}{\P(X_1>x)},
\end{equation}
then by Lemma 2.3 in \cite{embrechts:lambrigger:wuthrich:2009}, we have that 
\begin{equation}\label{e:Embrechts}
 \lim_{q\nearrow 1} \frac{VaR_q(S)}{VaR_q(X_1)} = \rho^{\xi}.
\end{equation}
The following result is extends the formulae in \cite{barbe:fougeres:genest:2006} (see also 
Theorem 4.1 of \cite{embrechts:lambrigger:wuthrich:2009}), which address only the case of equal 
portfolio weights and tail-equivalent losses.

\begin{prop}\label{p:Barbe-fla-extension}
Let $\bfZ = (Z_i)_{i=1}^d := (X_i^{1/\xi})_{i=1}^d \in RV_1(\{b_n\},\nu)$ be a non-negative regularly 
varying random vector with exponent equal to $1$.  Fix a norm $\|\cdot\|$ in $\R^d$ and let $\sigma_\bfZ$ be the spectral measure of $\bfZ$ induced on the 
positive unit sphere $\mathbb S_{\|\cdot\|}^+:= \{ \bfx \ge \mathbf 0\, :\, \|\bfx\| =1\}$. That is,
\begin{equation}\label{e:l:rho-nu}
 \nu(d\bfx) = c r^{-2}d r \sigma_\bfZ(d \bfu),
 \end{equation}
 where $c = \nu\{ \|\bfx\|>1\}$ 
 and $(r,u):= (\|\bfx\|, \bfx/\|\bfx\|)$ are the polar coordinates in $[0,\infty)^d\setminus\{\mathbf 0\}$.

For $\rho_{\bfw} = \rho(S,X_1)$ in \eqref{e:rho-S-X1}, we have
\begin{equation}\label{e:Barbe-fla-extension}
 \rho_{\bfw} (S,X_1) = \frac{1}{\sigma_1} \int_{S_{\|\cdot\|}^+} {\Big(} \sum_{i=1}^d w_i u_i^{\xi} {\Big)}^{1/\xi} \sigma_{\bfZ}(d\bfu),
\end{equation}
where  $\sigma_1:= \int_{S_{\|\cdot\|}^+} u_1 \sigma_{\bfZ} (d\bfu)$.
\end{prop}

The proof is a direct consequence of the next lemma, which establishes the asymptotic scale of $h(\bfZ)$ for 
a general $\xi$-homogeneous risk functional $h$.

\begin{lem} \label{l:rho} Let $\bfZ$ be as in Proposition \ref{p:Barbe-fla-extension} and  $h:[0,\infty)^d\to [0,\infty)$ be an arbitrary
non-negative $\xi-$homogeneous function, i.e.\ $h(c\bfx) = c^\xi h(\bfx),\ \forall c>0$. Then, for all $s>0$, we have
$$
 n\P(b_n^{-1/\xi} h(\bfZ) > s) \longrightarrow c \times \rho(h) s^{-1/\xi},\ \ \mbox{ as } n\to\infty, 
$$
where
\begin{equation}\label{e:rho-h}
\rho(h) = \int_{\mathbb S_{\|\cdot\|}^+} h(\bfu)^{1/\xi} \sigma_{\bfZ}(d\bfu).
\end{equation}
\end{lem} 
This result shows that $h(\bfZ)$ is regularly varying (provided $\rho(h)>0$) 
and in fact it identifies its asymptotic scale coefficient in terms of the spectral measure $H$.

\begin{proof}[Proof of Lemma \ref{l:rho}] By Theorem 6 and Remark 7 of \cite{hult:lindskog:2005}, we have that 
\begin{equation}\label{e:l:rho-1}
n\P(h(b_n^{-1}\bfZ) > s) \longrightarrow \nu\circ h^{-1} (s,\infty),\ \ \mbox{ as }n\to\infty.
\end{equation}
Note that the above convergence is valid for all $s>0$ since the by the scaling property of $\nu$ and the homogeneity of $h$, all sets $h^{-1}(s,\infty) 
= s^{1/\xi} h^{-1}(1,\infty)$ are in fact continuity sets of $\nu$. It remains to express the right-hand side of \eqref{e:l:rho-1} in terms of the spectral measure 
$\sigma_\bfZ$.
In view of \eqref{e:l:rho-nu} and by using the $\xi$-homogeneity of $h$, we obtain
\begin{eqnarray*}
\nu\circ h^{-1}(s,\infty) &=& c \int_{\mathbb S_{\|\cdot\|}^+} \int_0^\infty 1_{h^{-1}(s,\infty)}(r\bfu) r^{-2}d r \sigma_\bfZ(d\bfu)\\
&=& c \int_{\mathbb S_{\|\cdot\|}^+} \int_0^\infty 1_{(s,\infty)}(r^{\xi} h(\bfu) ) r^{-2}d r \sigma_\bfZ(d\bfu)\\
&=& c \int_{\mathbb S_{\|\cdot\|}^+} \int_0^\infty 1_{((s/h(\bfu))^{1/\xi}, \infty)} (r) r^{-2}dr \sigma_\bfZ(d\bfu)\\
&=& c \int_{\mathbb S_{\|\cdot\|}^+} (s/h(\bfu))^{-1/\xi} \sigma_\bfZ(d\bfu).
\end{eqnarray*}
The last expression equals $c \rho(h) s^{-1/\xi}$, where $\rho(h)$ is given in \eqref{e:rho-h}.
\end{proof}

\begin{rem} By using Lemma 2.3 of \cite{embrechts:lambrigger:wuthrich:2009} and our Lemma \ref{l:rho}, one can establish
the asymptotic value-at-risk for more complicated instruments, which are non-linear homogeneous 
functions of the underlying assets. For example, one can consider $h(\bfu):= \min_{i=1,\ldots,d} u_i^{\xi}$.  Thus, $h(\bfZ) = \min_{i=1,\ldots,d} X_i$ 
represents the minimum loss of a portfolio and bounds on its extreme VaR may be of interest.  Note that in this case
$$
 \lim_{x\to\infty} \frac{\P(\min_{i=1,\ldots,d} X_i >x) } {\P(X_1>x)} = \frac{1}{\sigma_1} \int_{S_{\|\cdot\|}^+} {\Big(} \min_{i=1,\ldots,d} u_i {\Big)} \sigma_\bfZ(d\bfu)
$$
does not depend on $\xi$.
\end{rem}

\subsection{On the role of the tail index in risk diversification} \label{sec:xi-trichotomy}

Here, we briefly comment on an intriguing {\em phase transition} in the Fr\'echet-type bounds for
the coefficient $\rho_{\bfw}$ in \eqref{eq:universal_rhobds_le1} occurring in the case when $\xi>1$.
Recall that extreme VaR equals $\rho_{\bfw}^\xi$, where $1/\xi$ is the tail exponent of the portfolio $\mathbf X$.  

The case $0<\xi<1$ corresponds to a finite-mean model for the losses.  In 
the case $\xi>1$, we have an infinite mean model, which may be viewed as
`catastrophic' since one has to have infinite capital in order to guard against such losses in the long-run.  
The bounds on $\rho_{\bfw}$ can be interpreted as follows:

\begin{itemize}
 \item In the light-tailed case $0<\xi<1$ the means of
the losses are {\em finite} and then the lower bound $\rho_{\bfw} = \sum_{i=1}^d w_i$ is achieved by the 
{\em asymptotically independent} portfolio.  This agrees with the general intuition that 
accumulating independent assets leads to {\em diversification}
and lower risk.  On the other hand, the worst case scenario, naturally, corresponds to perfect 
(asymptotic) dependence where all assets are asymptotically identical or no diversification at all.

\item In the boundary case $\xi=1$, the two bounds coincide, regardless of the asymptotic 
portfolio dependence.

\item In the extreme heavy-tailed setting $\xi>1$ the means of the losses are infinite and it turns out that 
the bounds in \eqref{eq:universal_rhobds_le1} are reversed.  Indeed, by the triangle inequality,
for the $L^{\xi}-$norm, we obtain:
\begin{align*}
\rho_{\bfw}  &= \int_{\mathbb S_+} \Big(w_1 u_1^\xi +\cdots+w_d u_d^\xi\Big)^{1/\xi} H(d\bfu) 
 \le \sum_{i=1}^d w_i^{1/\xi}  \int_{\mathbb S_+} u_i H(d\bfu)  =   \sum_{i=1}^d w_i^{1/\xi},
\end{align*}
where in the last relation we used the moment constraints in \eqref{eq:marginal_constraints}.
Thus, the expression for the lower bound in the case $0<\xi<1$ in \eqref{eq:universal_rhobds_le1} now, in the case $\xi>1$,
becomes the upper bound.

On the other hand, by the Jensen's inequality, for the concave function $x\mapsto x^{1/\xi}$, we have
$$
\left (w_1 u_1^\xi +\cdots+w_d u_d^\xi \right)^{1/\xi} \ge \Big(\sum_{i=1}^d w_i \Big)^{1/\xi}  \sum_{i=1}^d \widetilde w_i u_i,
$$
where $\widetilde w_i:= w_i/(\sum_{j=1}^d w_j)$, so that $\sum_{i=1}^d \widetilde w_i = 1$.  By integrating the last bound
with respect to $H(d\bfu)$, and using the moment constraints \eqref{eq:marginal_constraints}, we obtain
$$
\rho_{\bfw} \equiv \int_{\mathbb S_+} \left(w_1 u_1^\xi +\cdots+w_d u_d^\xi\right)^{1/\xi} H(d\bfu) \ge \Big(\sum_{i=1}^d w_i\Big)^{1/\xi}.
$$
This shows that the expression for the upper bound in \eqref{eq:universal_rhobds_le1} (for the case $0<\xi<1$) 
now (in the case $\xi>1$) yields the lower bound.

In summary, for the case $\xi>1$, we obtain the following universal bounds on $\rho_{\bfw}$ (see also \eqref{eq:universal_rhobds_le1})
$$
 \Big(\sum_{i=1}^d w_i\Big)^{1/\xi} \le \rho_{\bfw} \le  \sum_{i=1}^d w_i^{1/\xi}.
$$
The bounds are sharp.  The upper bound corresponds to asymptotic independence, and the lower to complete (asymptotic) dependence.
This contradicts with our intuition about diversification.  It shows that in the infinite-mean scenario, of potentially catastrophic
losses, it is best to just hold a single asset rather than to `diversify' among independent ones.
The following argument provides some explanation of this counter-intuitive phenomenon.

Let $X_i,\ i=1,2,\ldots,$ be non-negative independent and identically distributed random variables 
modeling losses. Suppose that $\P(X_i>x) \sim c x^{-1/\xi},\ x\to\infty,\ c>0$, with $\xi>1$ so that we are
in the extreme heavy tailed regime of infinite expected loss $\mathbb E (X_i) = \infty$.  
Suppose that unit investment is distributed evenly among $n$ such potentially catastrophic assets resulting in a portfolio loss 
$$S_n:= \frac{1}{n} \sum_{i=1}^n X_i.$$
Then, by the heavy-tailed version of the central limit theorem, we have
$$
\frac{1}{n^\xi} \sum_{i=1}^n X_i = \frac{S_n}{n^{\xi-1}} \stackrel{d}{\longrightarrow} Z,\ \ \mbox{ as }n\to\infty,
$$
where $Z$ is a non-trivial totally skewed, $(1/\xi)$-stable random variable \cite{samorodnitsky:taqqu:1994book}.
In this case, since $(\xi-1)>0$, the total loss $S_n \stackrel{d}{\approx} n^{\xi-1} Z$ stochastically grows to infinity
as the number of independent assets in the portfolio increases.  This counter-intuitive phenomenon where distributing an
investment among multiple independent assets is in fact detrimental is due the extreme heavy-tailed nature of 
the model.  Although such catastrophic models may not be practically relevant, the above argument shows that during regimes
of very extreme losses our intuition about diversification may fail.
\end{itemize}

\section{\label{sec:Proofs}Proofs}

\subsection{\label{subsec:KKT}Karush\textendash Kuhn\textendash Tucker Conditions}

The following proposition establishes sufficient conditions for optimal solutions to an LSIP $(P)$. 
This version of the classic Karush--Kuhn--Tucker
(KKT) optimality conditions  for the case of LSIPs will be used in the proofs for Theorems 
\ref{thm:sol_L_linprog}, \ref{thm:dvariate-lower} and \ref{thm:d-variate-upper}.

\begin{prop}[KKT conditions]
\label{prop:KKT_prime} Suppose Assumption \ref{as:main} is satisfied
and $\mathrm{val}(P)<\infty$. Fix $\boldsymbol{x}\in\mathbb{R}^{p}$.
If there exists dual variables $(y_1, y_2, \ldots , y_p)^\top\in\mathbb{R}_{+}^{p}$ and $\{t_{1},\ldots,t_{p}\}\subset T$
such that 
\begin{equation}
\sum_{k=1}^{p}y_{k}\boldsymbol{a}(t_{k})=\boldsymbol{c},\label{eq:KKT_dual_feasible-1}
\end{equation}
\begin{equation}
\boldsymbol{a}(t_{k})^{\top}\boldsymbol{x}=b(t_{k}),\ k=1,\ldots,p,\label{eq:KKT_comp_slack-1}
\end{equation}
and 
\begin{equation}
\boldsymbol{a}(t)^{\top}\boldsymbol{x}\ge b(t),\ \text{for all }t\in T.\label{eq:KKT_prime_feasible-1}
\end{equation}
Then $\boldsymbol{x}\in\mathrm{sol}(P)$. \end{prop}
\begin{proof}
 For every $\boldsymbol{x}\in\mathbb{R}^{p}$, define the set
of \emph{active indices} $T(\boldsymbol{x}):=\{t\in T:\boldsymbol{a}(t)^{\top}\boldsymbol{x}=b(t)\}.$
By Theorem 7.1.(ii) of \cite{goberna:lopez:1998} (see also Section 11.2
therein), a primal feasible vector $\tilde{\boldsymbol{x}}\in\mathbb{R}^{p}$
is optimal for $(P)$ whenever
\begin{equation}
\boldsymbol{c}\in\mathrm{cone}\left\{ \boldsymbol{a}(t):t\in T(\tilde{\boldsymbol{x}})\right\} ,\label{eq:active_cone}
\end{equation}
where $\mathrm{cone}\{C\}$ denotes the smallest convex cone containing
$C\subset\mathbb{R}^{p}$.  This is true in our setting. Indeed, Relation  \eqref{eq:KKT_comp_slack-1} 
implies that $\{t_1,\dots,t_p\}\subset T(\boldsymbol{x})$, which in view of \eqref{eq:KKT_dual_feasible-1} entails 
\eqref{eq:active_cone}.
\end{proof}

\subsection{Proof for the Tawn-Molchanov minimizer} 

In this section, let $D = \{1,\ldots,d\}$. Denote $2^{D}$ as the power set of
$D$ and $K^{c}=D\backslash K.$ We shall need two auxiliary lemmas.
\begin{lem}
\label{lem:complete_alternation} Let $0\le u_{(1)}\le u_{(2)}\le\cdots\le u_{(d)}\le1$
be the order statistics for arbitrary $\boldsymbol{u}\in\mathbb{S}_{+}^{d-1}$Fix
$\xi>0$ and define $u_{(0)}=0$. The following equality holds 
\begin{equation}\label{e:lem:complete_alternation}
\sum_{J\in2^{D}\backslash\varnothing}\max_{j\in J}\{u_{j}\}\sum_{L\subset J}(-1)^{|L|+1}|J^{c}\cup L|^{1/\xi}=\sum_{j=1}^{d}(d+1-j)^{1/\xi}\left(u_{(j)}-u_{(j-1)}\right).
\end{equation}
 \end{lem}
\begin{proof}
We will prove \eqref{e:lem:complete_alternation}  under the assumption that there are no ties, i.e., $u_{(1)} < u_{(2)} \cdots < u_{(d)}$. Since 
both the left- and right-hand sides of \eqref{e:lem:complete_alternation}  are continuous functions of the $u_{i}$'s, the general result will follow by
continuity for all $\mathbf u \in \mathbb S_+^{d-1}$.

We have
\begin{align}
&\sum_{J\in2^{D}\backslash\varnothing}\max_{j\in J}\{u_{j}\}\sum_{L\subset J}(-1)^{|L|+1}|J^{c}\cup L|^{1/\xi}\nonumber \\
&\quad\quad\quad=\sum_{i=1}^{d}u_{(i)}\left\{ \sum_{J\in2^{D}\backslash\varnothing}\mathbb{I}\left(\max_{j\in J}\{u_{j}\}=u_{(i)}\right)\sum_{k=0}^{|J|}{|J| \choose k}(-1)^{k+1}(d-|J|+k)^{1/\xi}\right\} \nonumber\\
&\quad\quad\quad=\sum_{i=1}^{d}u_{(i)}\left\{ \sum_{\ell=1}^{i}\sum_{{J\in2^{D}\backslash\varnothing\atop |J|=\ell}}\mathbb{I}\left(\max_{j\in J}\{u_{j}\}=u_{(i)}\right)\sum_{k=0}^{\ell}
{\ell \choose k}(-1)^{k+1}(d-|J|+k)^{1/\xi}\right\} \nonumber\\
&\quad\quad\quad=\sum_{i=1}^{d}u_{(i)}\left\{ \sum_{\ell=1}^{i}{i-1 \choose \ell-1}\sum_{k=0}^{\ell}{\ell \choose k}(-1)^{k+1}(d-(\ell-k))^{1/\xi}\right\} .\label{eq:comp_alt_1}
\end{align}
The second relation above follows from the fact that due to lack of ties, only sets $J$ containing at most $i$ indices will contribute to the 
inner sum therein. The last relation follows by a simple counting argument since ${ i-1 \choose \ell -1}$ is the number of sets $J$ with $|J|=\ell \le i$, for which
$\max_{j\in J} u_j = u_{(i)}.$  Indeed, due to lack of ties, the latter equality holds only if the set $J$ contains the (unique!) index of $u_{(i)}$ and $(\ell-1)$ other indices among those 
of $u_{(1)},\ldots,u_{(i-1)}$.
 
Now fix $i\in D$ and consider 
\begin{multline}
\sum_{\ell=1}^{i}{i-1 \choose \ell-1}\sum_{k=0}^{\ell}{\ell \choose k}(-1)^{k+1}(d-(\ell-k))^{1/\xi}\\
=d^{1/\xi}\sum_{k=1}^{i}{i-1 \choose k-1}(-1)^{k+1}+\sum_{q=1}^{i}(d-q)^{1/\xi}\sum_{k=0}^{i-q}{i-1 \choose q+k-1}{q+k \choose k}(-1)^{k+1}. \label{eq:comp_alt_2_new}
\end{multline}
By using the fact that ${ q+k \choose k} = {q+k-1 \choose k} + {q+k-1 \choose k-1}$, where by convention ${q+k-1 \choose k-1} =0$ if $k=0$, we obtain
$$
 {i-1 \choose q+k-1} { q+k \choose k}   =  { i-1 \choose q-1} {i-q \choose k} + {i-1 \choose q} {i-q-1 \choose k-1}.
$$
Now, by using the Newton's binomial expansion of $(1+(-1))^{i-q}$ and $(1+(-1))^{i-q-1}$, for the inner sum in the right-hand side of \eqref{eq:comp_alt_2_new},
we obtain that
\begin{eqnarray*}
\sum_{k=0}^{i-q} {i-1 \choose q + k-1} {q+k \choose k} (-1)^{k+1} & = & {i-1 \choose q-1} \sum_{k=0}^{i-q} {i-q \choose k} (-1)^{k+1} + {i-1 \choose q} \sum_{k=1}^{i-q} 
{i-q \choose k-1} (-1)^{k+1}\\
&=& (-1) \mathbb I(i-q=0) + \mathbb I(i-q=1) = (-1)^{i-q+1}\mathbb I(0\le i-q \le 1).
\end{eqnarray*}
By substituting in \eqref{eq:comp_alt_2_new}, we finally obtain
\begin{eqnarray}
\sum_{\ell=1}^{i}{i-1 \choose \ell-1}\sum_{k=0}^{\ell}{\ell \choose k}(-1)^{k+1}(d-(\ell-k))^{1/\xi}     &=& 
d^{1/\xi}\mathbb{I}\left(i=1\right)+\sum_{q=1}^{i}(d-q)^{1/\xi}(-1)^{i-q+1}\mathbb{I}\left(i-q\le1\right)\nonumber\\
&=&(d+1-i)^{1/\xi}-(d-i)^{1/\xi}.\label{eq:comp_alt_2}
\end{eqnarray}
Substituting \eqref{eq:comp_alt_2} into \eqref{eq:comp_alt_1} gives \eqref{e:lem:complete_alternation}, which completes the proof.
\end{proof}
The next lemma establishes analytical solutions to the dual of problem (${\cal L}_\rho$) in \eqref{eq:TM} in 
the case where the set of constraints includes the entire set of extremal coefficients
$\boldsymbol{\vartheta}=(\vartheta_{J})_{J\in2^{D}\backslash\varnothing}\in\mathbb{R}_{+}^{2^{d}-1}$:
\begin{eqnarray*}
({\cal L}_{\rho}^{\prime}(\boldsymbol{\vartheta}))\quad\quad\quad\quad\inf_{\boldsymbol{x}\in\mathbb{R}^{p}} &  & -\boldsymbol{\vartheta}^{\top}\boldsymbol{x}\\
\text{subject to:} &  & -\left( \left(u_{1}^{\xi}+\cdots+u_{d}^{\xi}\right)^{1/\xi} - \sum_{J\in2^{D}\backslash\varnothing}\max_{j\in J}\{u_{j}\}x_{J} \right) \le 0,\ \boldsymbol{u}\in\mathbb{S}_{+}^{d-1}.
\end{eqnarray*}
Observe that the dual to the minimization problem (${\cal L}_\rho$) is a 
maximization problem.  For convenience, we encode it equivalently as a minimization of the negative objective.

\begin{lem}
\label{lem:Istar_opt} The vector $\tilde{\boldsymbol{x}}=(\tilde{x}_{J})_{J\in2^{D}\backslash\varnothing}$
with elements 
\begin{equation}\label{e:x-tilde_J}
\tilde{x}_{J}:=\sum_{L\subset J}(-1)^{|L|+1}|J^{c}\cup L|^{1/\xi}
\end{equation}
 is optimal for Problem $({\cal L}_\rho^{\prime}(\boldsymbol{\vartheta}))$ with
 \[
 \mathrm{val}({\cal L}_\rho^{\prime}(\boldsymbol{\vartheta})) = \sum_{K\in2^{D}\backslash\varnothing}|K|^{1/\xi}\beta_{K},
\]
 where $(\beta_{K})_{K\in2^{D}\backslash\varnothing}\in\mathbb{R}_{+}^{p}$
is the unique solution to 
\begin{equation}
\sum_{K\in2^{D}\backslash\varnothing}\mathbb{I}\left\{ \left(J\cap K\right)\not=\varnothing\right\} \beta_{K}=\vartheta_{J},\ J\in2^{D}\backslash\varnothing.\label{eq:beta_sol}
\end{equation}
 \end{lem}
\begin{proof}
Fix $p=2^{d}-1.$ We prove $\tilde{\boldsymbol{x}}\in\mathrm{sol}({\cal L}_\rho^{\prime}(\boldsymbol{\vartheta}))$
by verifying the KKT optimality conditions of Proposition \ref{prop:KKT_prime}.  
That is, we need to show there exists $(y_{K})_{K\in2^{D}\backslash\varnothing}\in\mathbb{R}_{+}^{p}$
and $\{\boldsymbol{u}_{K} = (u_{jK})_{j=1}^d,\ K\in2^{D}\backslash\varnothing\}\subset\mathbb{S}_{+}^{d-1}$
such that the following conditions hold:
\begin{itemize}
\item[]{\em Dual feasibility:}
\begin{equation}
\sum_{K\in2^{D}}\max_{j\in J}\{u_{jK}\}y_{K}=\vartheta_{J},\ J\in2^{D}\backslash\varnothing,\label{eq:KKT_dual_feasible}
\end{equation}
\item[]{\em Complementary slackness:}
\begin{equation}
\sum_{J\in2^{D}\backslash\varnothing}\max_{j\in J}\{u_{jK}\}\tilde x_{J}=\left(u_{1K}^{\xi}+u_{2K}^{\xi}+\cdots+u_{dK}^{\xi}\right)^{1/\xi},\ K\in2^{D}\backslash\varnothing,\label{eq:KKT_comp_slack}
\end{equation}
\item[]{\em Primal feasibility:}
\begin{equation}
\left(u_{1}^{\xi}+u_{2}^{\xi}+\cdots+u_{d}^{\xi}\right)^{1/\xi}\ge 
\sum_{J\in2^{D}\backslash\varnothing}\max_{j\in J}\{u_{j}\}\tilde x_{J},\ \text{for all }\boldsymbol{u}\in\mathbb{S}_{+}^{d-1}.\label{eq:KKT_prime_feasible}
\end{equation}
\end{itemize}
Theorem 4 of \cite{schlather:tawn:2002} asserts that for a consistent set of extremal
coefficients Relation \eqref{eq:beta_sol} holds for some non-negative $\beta_{K},\ \varnothing\not=K\subset D$.
Define $y_{K}:=|K|\beta_{K}$ and $\boldsymbol{u}_{K}:=|K|^{-1}(\boldsymbol{1}_{K}(i))_{i=1}^{d}\in\mathbb{S}_{+}^{d-1}.$ 
We will show that the KKT conditions \eqref{eq:KKT_dual_feasible}-\eqref{eq:KKT_prime_feasible} hold.  This will complete the proof.

\begin{itemize}

\item[]{\em Dual feasibility \eqref{eq:KKT_dual_feasible}:} We have
\begin{align*}
\sum_{K\in2^{D}\backslash\varnothing}\max_{j\in J}\{u_{jK}\}y_{K}&=\sum_{K\in2^{D}\backslash\varnothing}\max_{j\in J}\{|K|^{-1}\boldsymbol{1}_{K}(j)\}|K|\beta_{K}\\
&=\sum_{K\in2^{D}\backslash\varnothing}\mathbb{I}\left\{ \left(J\cap K\right)\not=\varnothing\right\} \beta_{K}=\vartheta_{J},
\end{align*}
where the last equality follows from \eqref{eq:beta_sol}.

\item[] {\em Complementary slackness \eqref{eq:KKT_comp_slack}:} With $\tilde x_{J}$ as in \eqref{e:x-tilde_J}, we have
\begin{align*}
\sum_{J\in2^{D}\backslash\varnothing}\max_{j\in J}\{u_{jK}\}\tilde{x}_{J}&=\sum_{J\in2^{D}\backslash\varnothing}\mathbb{I}\left\{ \left(J\cap K\right)\not=\varnothing\right\} |K|^{-1}\sum_{L\subset J}(-1)^{|L|+1}|J^{c}\cup L|^{1/\xi}\\
&=|K|^{-1}\sum_{J\in2^{D}\backslash\varnothing}\mathbb{I}\left\{ \left(J\cap K\right)\not=\varnothing\right\} \sum_{L\subset J}(-1)^{|L|+1}|J^{c}\cup L|^{1/\xi}\\
&=|K|^{-1}|K|^{1/\xi}=\left(u_{1K}^{\xi}+u_{2K}^{\xi}+\cdots+u_{dK}^{\xi}\right)^{1/\xi},\ K\in2^{D}\backslash\varnothing.
\end{align*}
The third equality above follows from the M\"obius inversion formula
(see Theorem 4 of \cite{schlather:tawn:2002}) and the last one from the definition of the $u_{jK}$'s. 

\item[] {\em Primal feasibility \eqref{eq:KKT_prime_feasible}:}
For $(u_{1},\ldots,u_{d})^{\top}\in\mathbb{S}_{+}^{d-1},$ define
$f_{k}(j)=\mathbb{I}\{k\le j\}\left(u_{(k)}-u_{(k-1)}\right)$ where
\[
0=u_{(0)}\le u_{(1)}\le\cdots\le u_{(d)},
\]
 are the order statistics of $(0,u_{1},\ldots,u_{d})$.  Observe that $u_{(j)} = \sum_{k=1}^d f_k(j)$. Hence,
 \begin{align}
\left(u_{1}^{\xi}+\cdots+u_{d}^{\xi}\right)^{1/\xi}&=\left\{ \sum_{j=1}^{d}\Big( f_{1}(j) + \cdots + f_{d}(j) \Big)^{\xi} \right\} ^{1/\xi}\nonumber\\
&\ge\left\{ \sum_{j=1}^{d}f_{1}^{\xi}(j)\right\} ^{1/\xi}+\cdots+\left\{ \sum_{j=1}^{d}f_{d}^{\xi}(j)\right\} ^{1/\xi} \nonumber\\
&=\sum_{j=1}^{d}(d+1-j)^{1/\xi}\left(u_{(j)}-u_{(j-1)}\right),\label{eq:reverse-minkowski_198}
\end{align}
where the last relation follows from the definition of the $f_{k}(j)$'s and the bound follows from the reverse Minkowski inequality valid 
in the case $0<\xi\le 1$ (see, e.g., inequality No.\ 198 of \cite{hardy:1934}). 

Now,  Lemma \ref{lem:complete_alternation} implies that the the right-hand side of \eqref{eq:reverse-minkowski_198} equals
$$
\sum_{J\in2^{D}\backslash\varnothing}\max_{j\in J}\{u_{j}\}\sum_{L\subset J}(-1)^{|L|+1}|J^{c}\cup L|^{1/\xi} = \sum_{J\in2^{D}\backslash\varnothing}\max_{j\in J}\{u_{j}\}\tilde{x}_{J}.
$$
which in view of \eqref{eq:reverse-minkowski_198}, implies \eqref{eq:KKT_prime_feasible}.
\end{itemize}

Hence, $\tilde{\boldsymbol{x}}\in\mathrm{sol}({\cal L}_\rho^\prime(\boldsymbol{\vartheta}))$
and 
\begin{align*}
\mathrm{val}({\cal L}_\rho^\prime(\boldsymbol{\vartheta})) & =\boldsymbol{\vartheta}^{\top}\tilde{\boldsymbol{x}}=\sum_{J\in2^{D}\backslash\varnothing}\tilde{x}_{J}\sum_{K\in2^{D}\backslash\varnothing}\mathbb{I}\left\{ \left(J\cap K\right)\not=\varnothing\right\} \beta_{K}\\
&=\sum_{J\in2^{D}\backslash\varnothing}\sum_{K\in2^{D}\backslash\varnothing}\mathbb{I}\left\{ \left(J\cap K\right)\not=\varnothing\right\} \beta_{K}\sum_{L\subset J}(-1)^{|L|+1}|J^{c}\cup L|^{1/\xi}\\
&=\sum_{K\in2^{D}\backslash\varnothing}|K|^{1/\xi}\beta_{K}.
\end{align*}
This completes the proof of Lemma \ref{lem:Istar_opt}.
\end{proof}

\begin{proof}[{\bf Proof of Theorem \ref{thm:sol_L_linprog}}]
 Let $\mathcal{H}_{\boldsymbol{c}}$ denote the space of finite Borel
measures on $\mathbb{S}_{+}^{d-1}$ satisfying 
\[
\left\{ \int_{\mathbb{S}_{+}^{d-1}}\max_{j\in J}\{u_{j}\}H(d\boldsymbol{u})=c_{J}\right\} _{J\in\mathcal{J}}.
\]
 Likewise, denote $\mathcal{H}_{\boldsymbol{\vartheta}}$ as the space
of finite Borel measures on $\mathbb{S}_{+}^{d-1}$ satisfying
\[
\left\{ \int_{\mathbb{S}_{+}^{d-1}}\max_{j\in J}\{u_{j}\}H(d\boldsymbol{u})=\vartheta_{J}\right\} _{J\in2^{D}\backslash\varnothing}.
\]
Hence, we may write Problem $({\cal L}_\rho)$ as
\begin{align}
\mathrm{val}({\cal L}_\rho)& = \inf_{H\in\mathcal{H}_{\boldsymbol{c}}}\int_{\mathbb{S}_{+}^{d-1}}\left(u_{1}^{\xi}+\cdots+u_{d}^{\xi}\right)^{1/\xi}H(d\boldsymbol{u}) \nonumber \\
&=\inf_{\boldsymbol{\vartheta}\in\varTheta_{\boldsymbol{c}}}\left\{ \inf_{H\in\mathcal{H}_{\boldsymbol{\vartheta}}}\int_{\mathbb{S}_{+}^{d-1}}\left(u_{1}^{\xi}+\cdots+u_{d}^{\xi}\right)^{1/\xi}H(d\boldsymbol{u})\right\} ,\label{eq:tm_proof_1}
\end{align}
where $\varTheta_{\boldsymbol{c}}=\{\boldsymbol{\vartheta}\in\varTheta:\vartheta_{J}=c_{J},\text{ for all }J\in\mathcal{J}\}.$
(Recall $\varTheta$ is the space of consistent extremal coefficients).
Now Lemma \ref{lem:Istar_opt} together with strong duality for $({\cal L}_\rho^{\prime}(\boldsymbol{\vartheta}))$
imply
\begin{align}
\mathrm{val}({\cal L}_\rho^{\prime}(\boldsymbol{\vartheta})) & =
\inf_{H\in\mathcal{H}_{\boldsymbol{\vartheta}}}\int_{\mathbb{S}_{+}^{d-1}}\left(u_{1}^{\xi}+\cdots+u_{d}^{\xi}\right)^{1/\xi}H(d\boldsymbol{u})
 \nonumber \\
&=\sum_{K\in2^{D}\backslash\varnothing}|K|^{1/\xi}\beta_{K}, \label{eq:tm_proof_2}
\end{align}
where $(\beta_{K})_{K\in2^{D}\backslash\varnothing}\in\mathbb{R}_{+}^{p}$, with $p:= 2^d -1$, is the unique solution to 
\[
\sum_{K\in2^{D}\backslash\varnothing}\mathbb{I}\left\{ \left(J\cap K\right)\not=\varnothing\right\} \beta_{K}=\vartheta_{J},\ J\in2^{D}\backslash\varnothing.
\]
(Uniqueness follows by  M\"obius inversion, see e.g.\ Theorem 4 of \cite{schlather:tawn:2002}.)
 Substituting \eqref{eq:tm_proof_2} into \eqref{eq:tm_proof_1} gives
\begin{eqnarray*}
\mathrm{val}({\cal L}_\rho) & = & \inf_{\boldsymbol{\beta}\in\mathbb{R}_{+}^{p}}\sum_{J\in 2^D\backslash\varnothing}|J|^{1/\xi}\beta_{J},\\
 &  & \mathrm{subject\ to}:\left\{ \sum_{K\in 2^D\backslash\varnothing}\mathbb{I}\left\{ (K\cap J)\not=\varnothing\right\} \beta_{K}=c_{J}\right\} _{J\in\mathcal{J}},
\end{eqnarray*}
which completes the proof of Theorem \ref{thm:sol_L_linprog}.
\end{proof}

\subsection{Proofs for the closed form solutions in Section \ref{sec:closed_form}}


 \label{sec:Proofs:closed_form}

\medskip
\begin{proof}[{\bf Proof of Theorem \ref{thm:dvariate-lower}}] Let $k \in \{1,\dots,d-1\}$ be such that
\begin{equation}\label{e:k}
\frac{d}{k+1}\le \vartheta < \frac{d}{k}.
\end{equation}
That is, $B_k = [d (k+1)^{-1},dk^{-1})$ is the (unique) set in \eqref{e:d-variate-lower}, such that $\vartheta\in B_k$.  One can then write
\begin{equation}\label{e:lambda}
\vartheta = \lambda \frac{d}{k} + (1-\lambda) \frac{d}{k+1}, \ \ \mbox{ where  }\lambda = \frac{ \vartheta d^{-1} - (k+1)^{-1} }{k^{-1} - (k+1)^{-1}}\in [0,1). 
\end{equation}

In view of Theorem \ref{thm:sol_L_linprog}, the lower bound ${\rm val}(\mathcal L_\rho)$ is the value of
a standard linear program \eqref{eq:TM}.  This linear program is the {\em dual} to the following {\em primal} 
linear program:
\begin{align*}
& \mathop{\sup}_{\mathbf x = (x_J,\ J\in {\cal J}) \in \mathbb R^{p} } \mathbf c^\top \mathbf x\\
& \mbox{ subject to:\  } |K|^{1/\xi} \ge \sum_{J \in {\cal J}} \mathbb{I}\{J\cap K\not=\varnothing\}  x_J, \ \ \mbox{ for all $K\in 2^D\setminus\emptyset$, }
\end{align*}
where $D: = \{1,\dots,d\}$,
$$
\mathbf c = (1,\cdots,1,\vartheta)^\top\in\mathbb R^{d+1}\ \ \mbox{ and } \ \ {\cal J} = \{ \{1\},\dots,\{d\}, \{1,\dots,d\}\}.
$$
We will exhibit a primal feasible vector $\mathbf x = (x_J,\ J\in {\cal J})$ and a dual feasible vector $\boldsymbol{\beta} = (\beta_K,\ K \in 2^{D}\setminus\emptyset\}$, 
for which 
$$
\mathbf c^\top \mathbf x = \sum_{\emptyset \not= K\subset\{1,\dots,d\}} |K|^{1/\xi} \beta_K = L(\vartheta)
$$
with $L(\vartheta)$ as in  in \eqref{e:d-variate-lower}.  This, will complete the proof by the strong 
duality between the standard linear programs.

{\em Primal vector.}  For each $J\in\mathcal{J}$, let
\begin{equation}\label{e:xJ}
x_{J}=\begin{cases}
(k+1)^{1/\xi}-k^{1/\xi} &,\ \  |J|=1\\
(k+1) k^{1/\xi}-k(k+1)^{1/\xi} &,\ \  |J|=d,
\end{cases}
\end{equation}
where $k$ is as in \eqref{e:k}.

{\em Dual vector.} Now, define the components of the dual vector as:
\[
\beta_{K}=\begin{cases}
 \frac{\lambda d}{k} {d \choose k}^{-1} & |K|=k\\
\frac{(1-\lambda) d}{(k+1)}{d \choose k+1}^{-1} & |K|=k+1\\
0 & |K|\not\in\{k,k+1\},
\end{cases}
\]
where $\lambda$ is defined in \eqref{e:lambda}.

{\em Dual feasibility.} We will first verify that $\boldsymbol{\beta}$ is dual feasible.  We need to verify, that for all $J\in {\cal J}$,
\begin{equation}\label{e:dual-feasible-beta-d-variate-lower}
\sum_{K \subset D,\ K\not =\varnothing}\mathbb{I}\{K\cap J\not=\varnothing\}\beta_{K}=c_J =\begin{cases}
1 & |J|=1\\
\vartheta & |J|=d.
\end{cases}
\end{equation}
Indeed, when $|J|=d$ (i.e., $J=\{1,\dots,d\}$) we have
\begin{align*}
\sum_{K\subset D,\ K\not=\varnothing}\mathbb{I}\{K\cap J\not=\varnothing\}\beta_{K}&=\frac{\lambda d}{k} \sum_{{K\subset D\atop |K|=k}}{d \choose k}^{-1}
 +\frac{(1-\lambda)d}{(k+1)}\sum_{{K\subset D\atop |K|=k+1}} {d \choose k+1}^{-1}\\
&= \lambda \frac{d}{k} + (1-\lambda) \frac{d}{k+1} = \vartheta,
\end{align*}
in view of \eqref{e:lambda}.
Let now $|J|=1,$ that is, $J= \{j\}$, for some arbitrary fixed $j\in D$. Then,
\begin{align*}
\sum_{K\subset D, K\not=\varnothing}\mathbb{I}\{K\cap\{j\}\not=\varnothing\}\beta_{K} & =\frac{\lambda d}{k} \sum_{K\subset D\, :\, j \in K\atop |K| = k} {d \choose k}^{-1} +  
\frac{(1-\lambda) d}{k+1} \sum_{K\subset D\, :\, j \in K\atop |K| = k+1} {d \choose k+1}^{-1} \\
&= \lambda \frac{d}{k}{d \choose k}^{-1} {d-1 \choose k-1}  + (1-\lambda) \frac{d}{k+1}{d \choose k+1}^{-1} {d-1 \choose k} \\
& = \lambda + (1-\lambda) =1.
\end{align*}
This completes the proof of \eqref{e:dual-feasible-beta-d-variate-lower}, i.e., the dual feasibility of $\boldsymbol{\beta}$.

\medskip
{\em Primal feasibility.} For all $\emptyset \not = K \subset\{1,\dots,d\}$, we need to show 
\begin{equation*}
 |K|^{1/\xi} \ge \sum_{J\in\mathcal{J}}\mathbb{I}\{K\cap J\not=\varnothing\}x_{J} 
\end{equation*}
Since $\xi\in (0,1)$, the function $t\mapsto t^{1/\xi}$ is convex on $t\in (0,\infty)$ and
hence for any $s$ and $t_{1}\le t_{2}\in\mathbb{R}_{+}$ such that $s\not\in(t_{1},t_{2})$
it follows that 
\[
s^{1/\xi}\ge\frac{t_{2}^{1/\xi}-t_{1}^{1/\xi}}{t_{2}-t_{1}}(s-t_{1})+t_{1}^{1/\xi}.
\]
We shall apply this inequality with $t_1:= k,\ t_2 := k+1$ and $s:= |K| \not \in (k,k+1)$.
(Note that $|K|$ is an integer, and hence we always have $|K|\not\in(k,k+1)$.)  We have:
\begin{align*}
|K|^{1/\xi} &\ge\frac{(k+1)^{1/\xi}-k^{1/\xi}}{k+1-k}(|K|-k)+k^{1/\xi}\\
&=|K|\left[(k+1)^{1/\xi}-k^{1/\xi}\right]+k^{1/\xi}-k\left[(k+1)^{1/\xi}-k^{1/\xi}\right]\\
&=|K|\left[(k+1)^{1/\xi}-k^{1/\xi}\right]+\frac{k^{1/\xi-1}-(k+1)^{1/\xi-1}}{k^{-1}-(k+1)^{-1}}\\
&=\sum_{J\in\mathcal{J}}\mathbb{I}\{K\cap J\not=\varnothing\}x_{J},
\end{align*}
where the last equality follows from \eqref{e:xJ}, since there are precisely $|K|$ singleton sets $J \in {\cal J}$ with $K\cap J \not=\emptyset$.
This establishes the primal feasibility of $\mathbf x = (x_J, J\in \mathcal{J})$.

{\em Optimality.} Finally, we will verify that the objective functions of the primal and dual linear programs coincide.
In view of \eqref{e:lambda}, with straightforward manipulations, we obtain
\begin{align}\label{e:d-variate-primal-objective}
\mathbf c ^\top \mathbf x &= \sum_{J\in {\cal J}\atop |J|=1} \left[ (k+1)^{1/\xi} - k^{1/\xi}\right] + \vartheta  \left[(k+1) k^{1/\xi} - k (k+1)^{1/\xi}\right] \nonumber\\
& = d (k+1)^{1/\xi} - d k^{1/\xi} + \lambda d \left[ (k+1) k^{1/\xi-1} - (k+1)^{1/\xi} \right]  \nonumber\\
&  \quad\quad + (1-\lambda) d  \left[ k^{1/\xi} - dk (k+1)^{1/\xi-1} \right] \nonumber\\
& = d (k+1)^{1/\xi-1} + \lambda d \left[ k^{1/\xi-1} - (k+1)^{1/\xi -1} \right] \nonumber\\
& = d\left\{ \lambda k^{1/\xi-1}+(1-\lambda)(k+1)^{1/\xi-1}\right\} =L(\vartheta). 
\end{align}

Next, we consider the value of the dual objective.  We have,
\begin{align}\label{e:d-variate-dual-objective}
\sum_{K\subset D}|K|^{1/\xi}\beta_{K}
&=\sum_{K\subset D}|K|^{1/\xi}\beta_{K}\nonumber\\
&=d\lambda\sum_{{K\subset D\atop |K|=k}}k^{1/\xi-1}{d \choose k}^{-1}
+d(1-\lambda)\sum_{{K\subset D\atop |J|=k+1}}(k+1)^{1/\xi-1}{d \choose k+1}^{-1}\nonumber \\
&=d\left\{ \lambda k^{1/\xi-1}+(1-\lambda)(k+1)^{1/\xi-1}\right\} =L(\vartheta).
\end{align}
Relations \eqref{e:d-variate-primal-objective} and \eqref{e:d-variate-dual-objective} show that the values of the
primal and dual objectives are both equal to $L(\vartheta)$  in \eqref{e:d-variate-lower}.
This completes the proof of Theorem \ref{thm:dvariate-lower}.
\end{proof}

\begin{proof}[{\bf Proof of Theorem \ref{thm:d-variate-upper}}]

We need the following elementary result.

\begin{lem}\label{l:ua} Let $0<\xi<1$, $c>0$ and $u_c(\vartheta):= (\vartheta^\xi + c \cdot (d-\vartheta)^\xi)^{1/\xi}$.
Then, for all $\vartheta\in [0,d]$, we have:

{\bf (i)} $u_c(\vartheta) - \vartheta u_c'(\vartheta)\ge 0$

{\bf (ii)} $u_c''(\vartheta) \le 0$ 

{\bf (iii)} For all $z,z'\in [0,d]$, we have $u_c(z')  \le u_c'(z)  (z'-z) + u_c(z).$
\end{lem}
\begin{proof} Parts {\bf (i)} and {\bf (ii)} can be verified with straightforward differentiation.  Part {\bf (ii)} implies 
that the function $u_c$ is concave, which entails part {\bf (iii)}.
\end{proof}

Recall the primal-dual correspondence established in Theorem \ref{thm:LSIPexists} between the problems $({\cal U}_\rho)$ and 
$({\cal U}'_\rho)$. That is, problem $({\cal U}_\rho)$ is the {\em dual} of the LSIP problem $({\cal U}_\rho')$ in \eqref{e:U_rho-prime}.

We call problem $({\cal U}_\rho')$ `primal' and $({\cal U}_\rho)$ `dual'.  We will construct a {\em primal feasible} 
vector $\mathbf x \in \mathbb R^p$ and a {\em dual feasible} measure $H$, such that 
\begin{equation}\label{e:val-primal-equals-val-dual}
v:= \mathbf c^\top \mathbf x = \int_{\mathbb S_+} (u_1^\xi + \cdots + u_d^\xi)^{1/\xi} H(du),
\end{equation}
then $v = {\rm val}({\cal U}_\rho') = {\rm val}({\cal U}_\rho)$ will be the (common) optimal value of the two problems.  

Let $p=d+1$ and $D= \{1,\dots,d\}$ and define the measure $H(d\mathbf u) = \sum_{k=1}^d \delta_{\mathbf u_k}(d\mathbf u)$, where $\mathbf u_k = (u_{jk})_{j=1}^d$ are such that
$$
u_{kk} = \frac{\vartheta}{d}\ \ \mbox{ and }\ \ u_{jk} = \frac{d-\vartheta}{d(d-1)},\ j\in D \setminus\{k\}.
$$
Notice that $\mathbf u_k \in \mathbb S_+$ and also the measure $H$ is dual feasible.  Indeed,
$$
\int_{\mathbb S_+} u_j H(d\mathbf u) = \sum_{k=1}^d u_{jk}  = 1,
$$
which shows that the marginal extremal index constraints are met.
On the other hand, since 
$$
\frac{\vartheta}{d} \ge \frac{d-\vartheta}{d(d-1)},\ \ \mbox{ for all }1\le  \vartheta \le d,
$$
for each $k$, we have $\max_{j\in D} u_{jk} = \vartheta/d$.  This implies that
$$
\int_{\mathbb S_+} \max_{j\in D} u_j H(d\mathbf u) = \sum_{k=1}^d \frac{\vartheta}{d}  = \vartheta,
$$
and hence the $d$-variate extremal index constraint is satisfied.
We have thus shown that the measure $H$ is dual feasible, i.e., meets the constraints of $({\cal U}_\rho)$.

Let now $z\in [1,d]$ and consider the function 
$$
U( z ) := \max_{\mathbf u \in \mathbb S_+,\ d \max_{j\in D} u_j = z } (u_1^\xi + \cdots + u_d^\xi)^{1/\xi}.
$$
A straightforward calculation using Lagrange multipliers  yields that
$$
U(z) = \frac{1}{d} \left( z^\xi + (d-1)^{1-\xi} (d-z)^\xi \right)^{1/\xi},\ \ z\in [1,d].
$$

Observe that for all $\mathbf u_k$ in the support of $H$, we have
$$
(u_{1k}^\xi + \cdots + u_{dk}^\xi)^{1/\xi} = U(\vartheta).
$$
Therefore, the value of the dual problem at $H$ is:
\begin{equation}\label{e:dual-value}
\int_{\mathbb S_+} (u_{1}^\xi + \cdots + u_{d}^\xi)^{1/\xi} H(d\mathbf u) = \sum_{k=1}^d U(\vartheta)= d U(\vartheta) 
= (\vartheta^\xi + (d-1)^{1-\xi} (d-\vartheta)^{\xi})^{1/\xi}.
\end{equation}

Let us now deal with the primal problem.  Consider the vector $\mathbf x = (x_i)_{i=1}^p$, where 
\begin{align*}
x_1 = \cdots = x_{d}  =  U(\vartheta)-\vartheta U^{\prime}(\vartheta),\ \ \mbox{ and } \ \ 
x_{d+1}  =  dU^{\prime}(\vartheta).
\end{align*}
We will show that $\mathbf x$ is primal feasible.  That is, with $\mathbf a (\mathbf u) = (u_1,\cdots,u_d,\max_{j\in D} u_j)^\top $ 
and $b(\mathbf u) = (u_1^\xi + \cdots + u_d^\xi)^{1/\xi}$, we have
$$
b(\mathbf u) \le \mathbf a(\mathbf u)^\top \mathbf x,\ \ \mbox{ for all }\mathbf u\in \mathbb S_+.
$$
Observe that by the definition of the function $U$, we have
\begin{equation}\label{e:b-le-U}
b(\mathbf u) \le U( d \max_{j\in D} u_j ),\ \ \mbox{ for all } \mathbf u = (u_j)_{j=1}^d\in \mathbb S_+.
\end{equation}
Now, by applying Lemma \ref{l:ua}.{\bf (iii)}, to $ u_c(z) = U(z)$ with $c := (d-1)^{1-\xi}$,  $z:= \vartheta$ and 
$z':= d\max_{j=1,\dots,d} u_j$, we obtain that 
\begin{align}\label{e:b-bound-RHS}
U( d \max_{j=1,\dots,d} u_j ) &\le U'(\vartheta) \left[ d \max_{j\in D} u_j -\vartheta \right] + U(\vartheta) \nonumber \\
& = \sum_{j=1}^d u_j (U(\vartheta) - \vartheta U'(\vartheta)) + \max_{j\in D} u_j  dU'(\vartheta) \nonumber \\
& = \sum_{j=1}^d u_j x_j + \max_{j\in D} u_j x_{d+1} \equiv \mathbf a (\mathbf u)^\top \mathbf x.
\end{align}
Since the last inequality is true for all $\mathbf u\in \mathbb S_+$, Relations \eqref{e:b-le-U} and \eqref{e:b-bound-RHS},
imply the primal feasibility of the point $\mathbf x$.

Finally, we compute the value of the primal objective at $\mathbf x$:
\begin{align*}
\mathbf c^\top \mathbf x &= \sum_{j=1}^d 1\times x_j + \vartheta \times x_{d+1} \\
& = d \times (U(\vartheta)-\vartheta U^{\prime}(\vartheta)) + \vartheta \times dU^{\prime}(\vartheta) = d U(\vartheta),
\end{align*}
which in view of \eqref{e:dual-value} coincides with the evaluation of the dual problem objective at the measure $H$.
This completes the proof of Theorem \ref{thm:d-variate-upper}
\end{proof}

\bibliographystyle{plain}

\end{document}